\documentclass[11pt,oneside,letterpaper]{article}
\usepackage{amsmath}
\usepackage{amsfonts}
\usepackage{amssymb}
\usepackage{amsthm}
\usepackage{graphicx}
\usepackage{multirow}
\usepackage{color}
\usepackage{rotfloat}
\usepackage[T1]{fontenc}
\usepackage{amssymb,amsmath}
\usepackage[left=1in, bottom=1in, right=1in, top=1in]{geometry}
\usepackage{times}
\usepackage{natbib}
\usepackage{cases}
\usepackage{inputenc}
\bibpunct{(}{)}{;}{a}{,}{,}
\usepackage{float}
\usepackage{pdfsync}
\usepackage{setspace}
\usepackage{enumerate}
\usepackage{mathtools}
\numberwithin{equation}{section}
\newtheorem{definition}{Definition}
\newtheorem{lemma}{Lemma}[section]
\newtheorem{proposition}{Proposition}[section]
\newtheorem{theorem}{Theorem}
\newtheorem{assumption}{Assumption}
\newtheorem{corollary}{Corollary}[section]
\theoremstyle{definition}
\theoremstyle{definition}\newtheorem{remark}{Remark}[section]
\DeclareMathOperator*{\argmin}{arg\,min}
\DeclareMathOperator*{\arginf}{arg\,inf}

\title{A Sieve M-Estimator for Entropic Optimal Transport}
\author{Rami V. Tabri \\
 Department of Econometrics and Business Statistics, Monash University, \\ Clayton, Victoria, Australia, \\ Email: rami.tabri@monash.edu.}
\begin{document}
\maketitle
\begin{abstract}
Entropically regularized optimal transport between probability measures supported on compact subsets of Euclidean space admits a representation as an information projection under moment inequality constraints. Exploiting this structure, I develop a sieve-based approximation of the Fenchel dual, yielding a sequence of finite-dimensional convex programs whose sample analogues provide tractable estimators of the regularized optimal value and associated dual optimizers. Under minimal assumptions—compact support and continuity of the cost function—I establish almost sure consistency of these estimators. I further derive finite-sample bounds for the estimation error of the optimal value, featuring only logarithmic dependence on sieve complexity, and obtain asymptotic stochastic bounds characterized by suprema of centered Gaussian processes. The results furnish general statistical guarantees for sieve-based estimation of entropic optimal transport and apply to settings not covered by existing theory for the empirical Sinkhorn divergence and other sieve-based methods.
\end{abstract}

\section{Introduction}
Optimal transport (\citealp{villani2009optimal}) has emerged as a fundamental tool in modern statistics, supporting methodologies for distributional comparison, generative modeling, domain adaptation, and causal inference. Given two Borel probability spaces, $(\mathcal{X},\mathcal{B}(\mathcal{X}),P_X)$ and $(\mathcal{Y},\mathcal{B}(\mathcal{Y}),P_Y)$, with $\mathcal{X}\subset\mathbb{R}^{d_x}$ and $\mathcal{Y}\subset\mathbb{R}^{d_y}$, and a cost function $c:\mathcal{X}\times\mathcal{Y}\rightarrow \mathbb{R}_+$, the optimal transport problem is an infinite-dimensional linear optimization problem
\begin{align}\label{eq - OTP}
\inf_{P\in\Pi(P_X,P_Y)}\int_{\mathcal{X}\times\mathcal{Y}}c(x,y)\,dP,
\end{align}
where the infimum is taken over the set $\Pi(P_X,P_Y)$ of couplings between $P_X$ and $P_Y$. Central to its recent success in practice is the paradigm of entropic regularization, popularized by~\cite{Cuturi-2013}, enabling computational advances
for data-rich applications in areas like machine learning or image processing (\citealp{PeyreCuturiComputationalOT}). For $\gamma>0$, the entropically regularized optimal transport (EOT) problem is 
\begin{align}\label{eq - Entropc OTP}
\inf_{P\in\Pi(P_X,P_Y)}\int_{\mathcal{X}\times\mathcal{Y}}c(x,y)\,dP+\gamma^{-1} H(P|P_X\otimes P_Y),
\end{align} 
where $\gamma$ is the entropic regularization parameter, and  $H(\cdot|P_X\otimes P_Y)$ denotes the relative entropy with respect to the product measure
\begin{align*}
H(P|P_X\otimes P_Y)=\begin{cases}
\int_{\mathcal{X}\times\mathcal{Y}}\frac{dP}{d(P_X\otimes P_Y)}\log\left(\frac{dP}{d(P_X\otimes P_Y)}\right)\,d(P_X\otimes P_Y) &\mbox{if } P\ll P_X\otimes P_Y\\
\infty &\mbox{elsewhere}.
\end{cases}
\end{align*}
For fixed $\gamma$, the problem (\ref{eq - Entropc OTP}) yields a regularized optimal value, dual potentials, and a transport plan that approximate those of
the problem~(\ref{eq - OTP}), with large values of $\gamma$ providing more accurate approximations (e.g.,~\citealp{LEONARD}). 

\par Entropic regularization enables scalable computation and smooths the geometry of the transport problem, and it is now routinely used in statistical practice via the empirical Sinkhorn divergence. Given samples, $X_1,\ldots,X_N$ i.i.d. from $P_X$ and 
$Y_1,\ldots,Y_N$ i.i.d. from $P_Y$, and their empirical distributions, $\hat{P}_X$ and $\hat{P}_Y$, respectively, the empirical Sinkhorn divergence solves the sample analogue of the EOT problem~(\ref{eq - Entropc OTP}): it replaces the population distributions $P_X$ and $P_Y$ in the EOT problem with $\hat{P}_X$ and $\hat{P}_Y$, respectively, and then applies the Sinkhorn algorithm to approximate the regularized optimal value, transport plan, and dual potentials. Despite its widespread use, relatively little is known about the statistical properties of the EOT problem beyond the narrow, though important, class of smooth cost functions. 

\par Existing theory offers only partial answers: central limit theorems are available for quadratic costs and sub-Gaussian probability measures (e.g.,~\citealp{delBarrioGonzalezSanzLoubesNilesWeed2023,GoldfeldKatoRiouxSadhu2024}), and more general cost functions have been analyzed only in discrete and semi-discrete settings (e.g.,~\citealp{BigotKlein2019,Bercu-Bigot}). Furthermore, sample-complexity bounds have been established under smooth costs and compactly supported measures (e.g.,~\citealp{GenevayCuturiPeyreBach2019}) and for quadratic costs with sub-Gaussian measures (e.g.,~\citealp{MenaNilesWeed2019}), ruling out many cost functions used in applications such as economics (\citealp{GalichonMongeAmpere,ChiapporiSalanieMatching}), imaging (\citealp{PeyreCuturiComputationalOT}), and machine learning (\citealp{Arjovsky2017WGAN,FlamaryCourtyOTML}). On the other hand, methods based on reproducing kernel Hilbert space (RKHS) parameterizations (\citealp{Genevey}) or neural networks (\citealp{Seguy}) are computationally flexible and accommodate general continuous cost functions, but impose continuity of the dual potentials as a structural assumption, motivated by optimization considerations. These methods lack statistical guarantees beyond consistency, which holds only under the correct specification of the potentials' continuity. As a result, the statistical behavior of EOT estimators under general continuous costs and general marginal distributions remains essentially unknown. This paper addresses this gap in the setting of compactly supported marginal distributions.

\par I develop a nonparametric estimation framework for the EOT value under minimal conditions: the cost function is continuous and the marginal distributions have compact support. The method is grounded in a representation of the EOT problem as an information projection under moment constraints--a classical perspective (e.g.,~\citealp{csiszar1975}) that, to the best of my knowledge, has not previously been exploited for statistical estimation of entropic optimal transport. This formulation enables the use of sieve methods and empirical process theory to approximate the Fenchel dual of the information-projection problem. The resulting estimator is computationally tractable and applies to continuous cost functions with arbitrary compactly supported marginal distributions, thereby extending statistical optimal transport to settings not covered by existing theory.

\par The present analysis develops a stochastic approximation scheme for the EOT problem~\eqref{eq - Entropc OTP} under the following minimal conditions.
\begin{assumption}\label{Assump -Primities OTP}
(i) $\mathcal{X}$ and $\mathcal{Y}$ are compact, (ii) $c:\mathcal{X}\times\mathcal{Y}\rightarrow\mathbb{R}_+$ is continuous.
\end{assumption}
\noindent Under Assumption~\ref{Assump -Primities OTP}, the EOT problem~\eqref{eq - Entropc OTP} admits a unique solution, denoted by $P_\gamma$. Rewriting the objective as
\begin{align}\label{eq - OT Obj rewrite}
\gamma^{-1}\log a_\gamma+\gamma^{-1}H(P\mid R_\gamma),
\end{align}
where the reference probability $R_\gamma$ satisfies
\begin{align*}
\frac{dR_\gamma}{d(P_X\otimes P_Y)}(x,y)\coloneqq \exp\left\{-\gamma c(x,y)\right\}/a_\gamma,\quad (x,y)\in\mathcal{X}\times\mathcal{Y} 
\end{align*}
with $a_\gamma\coloneqq\int_{\mathcal{X}\times\mathcal{Y}}\exp\left\{-\gamma c(x,y)\right\}\,d(P_X\otimes P_Y)(x,y)$,
reveals that $P_\gamma$ is also the solution to the information projection ($I$-projection) problem
\begin{align}\label{eq - Static Schrodinger prob}
\min\left\{H(P\mid R_\gamma):P\in\Pi(P_X,P_Y)\right\}.
\end{align}

\par The stochastic approximation scheme developed in this paper applies the \emph{method of sieves} (\citealp{Grenander81}) to the Fenchel dual of the $I$-projection problem~\eqref{eq - Static Schrodinger prob}. The construction rests on three main ideas:
\begin{enumerate}[(a)]
\item reformulating the marginal constraints in~\eqref{eq - Static Schrodinger prob} as moment inequality restrictions;
\item reparametrizing the Fenchel dual so that it admits finite-dimensional convex approximations; and
\item estimating these finite-dimensional programs via the \emph{Sample Average Approximation} method (\citealp{Shapiro-Dentcheva-Ruszczynski}).
\end{enumerate}
Together, these steps yield a sieve M-estimator for the Fenchel dual problem. The approximation framework of \cite{Tabri-MOOR-2025} provides a convergent sieve construction for implementing steps (a) and (b). The key insight behind step (a) is that cumulative distribution functions characterize probability measures supported on subsets of Euclidean space, allowing us to rewrite moment equality constraints as pairs of moment inequalities. This reveals that the resulting class of moment functions is uniformly bounded and Vapnik-Chervonenkis (VC), and hence suitably precompact for adapting the general results of \cite{Tabri-MOOR-2025}. Unlike existing approaches based on discretization or parametric approximation of dual potentials, this construction enforces the defining constraints of the EOT problem at the population level and uses the sample only to approximate expectations in the dual objective.

\par That framework reparametrizes the Fenchel dual so that the dual variable becomes a \emph{Gelfand–Pettis integral} (\citealp{Gelfand}; \citealp{Pettis}): a weak vector-valued integral with respect to a positive Radon measure supported on the $L_1(R_\gamma)$-closed convex hull of the moment functions defining the inequality constraints. This representation is advantageous because it admits approximation by Riemann sums in the $L_1(R_\gamma)$ norm, forming the basis of the sieve construction in \cite{Tabri-MOOR-2025}. Specializing this scheme to the present setting produces a sequence of finite-dimensional convex stochastic programs with a common objective and expanding domains. Here, these domains are finite-dimensional spaces generated by selected moment functions, which remain VC because they consist of linear combinations of elements from a VC class.

To apply the SAA method, I begin by solving sample analogues of the sieve-based approximating programs. Because their domains are naturally unbounded, I restrict them in a way that still allows expansion with the sieve approximation. I justify this restriction by connecting the Fenchel dual problem's optimizers to Schr\"{o}dinger potentials (\citealp{nutz2021entropic}), and I establish that these optimizers remain uniformly bounded under Assumption~\ref{Assump -Primities OTP}. With these restrictions, I show that when the sieve grows slowly enough relative to the sample size, a strong uniform law of large numbers (ULLN) holds for the corresponding empirical processes. This ULLN guarantees that optimizing the sample-based objective over the sieve is asymptotically equivalent to optimizing its expectation, the Fenchel dual objective. I provide the precise statement and proof of this ULLN in Appendix~\ref{Appendix - ULLN}. Notably, it accommodates sequences of function classes whose index sets evolve with sample size--a situation not addressed by existing ULLN results, to the best of my knowledge.

\subsection{Summary of Contributions}

\par The contributions of this paper are threefold.

\smallskip

\noindent 1. I establish the almost sure convergence of the proposed stochastic approximation scheme. In particular, under random sampling from $R_\gamma$ with sample size $N$,
\begin{enumerate}[(i)]
\item the optimal value of the dual problem is consistently estimated by the optimal values of a sequence of finite-dimensional SAA convex programs, and
\item every weak-star accumulation point of the corresponding sequence of optimal solutions is an optimal solution of the dual problem,
\end{enumerate}
as $N\to\infty$, with probability 1. 

\smallskip

\noindent 2. Using empirical process techniques, I derive a finite-sample rate in terms of mean-convergence for the estimator of the EOT value. The rate separates stochastic variation from sieve approximation error, yielding a nonparametric bias-variance trade-off. The stochastic component is controlled by Massart's finite class lemma (see \citealp{Massart2000}, and Theorem~3.3 in~\cite{MohriRostamizadehTalwalkar2018}), giving a bound $\sqrt{2\log(n)/N}$, where $n$ is the dimension of the approximating sieve function class. This yields rates slower than the classical $\sqrt{N}$ benchmark--typical in nonparametric estimation--while exhibiting unusually mild dependence on model complexity due to the logarithmic factor.

\smallskip

\noindent 3. I establish matching upper and lower stochastic bounds for the estimation error of the dual problem's optimal value, with leading terms given by the suprema of centered Gaussian processes indexed by the sieve class, divided by $\sqrt{N}$. These results follow from the Gaussian approximation of suprema of empirical processes in \cite{chernozhukov2014gaussian} and can be used to construct interval estimators for the EOT value. Importantly, all results require only continuity of the cost and compact supports, making the methodology applicable in settings for which no statistical theory currently exists.

\subsection{Organization of the paper}

Section~\ref{Section Framing as MIs} develops the moment-inequality
representation of the EOT problem and
establishes key structural properties of the associated Fenchel dual problem via a sequence of
propositions. These results lay the analytical foundation for the statistical
analysis. Section~\ref{Section Sieve M-Estimation} introduces the sieve M-estimator and states the main result of the paper, Theorem 1, which establishes consistency, the convergence rate, and stochastic bounds. Section~\ref{Section Discussion} discusses the scope of the main result, its
implications for statistical optimal transport, and potential extensions. The Appendix collects proofs of the theorem and all propositions to streamline the exposition, as well as all technical lemmas, derivations, and a numerical illustration.

\section{Moment Inequality Information Projection Problem and Approximation}\label{Section Framing as MIs}
 Section 3 of~\cite{csiszar1975} frames the $I$-projection problem~(\ref{eq - Static Schrodinger prob}) in terms of moment equality constraints. He argues that two probability measures on abstract spaces are equal if and only if they have equal moments for \emph{all} integrable moment functions. Since I focus on the setup specified by Assumption~\ref{Assump -Primities OTP}, I can simplify this characterization. The constraint set $\Pi(P_X,P_Y)$ has the following characterization:
\begin{align}\label{eq - constraint set Borel}
\left\{P\in\mathcal{P}\left(\mathcal{X}\times\mathcal{Y}\right):P_1(A)=P_X(A)\,\forall A\in\mathcal{B}(\mathcal{X})\,\text{and}\,P_2(A)=P_Y(A)\,\forall A\in\mathcal{B}(\mathcal{Y})\right\},
\end{align}
where $\mathcal{P}\left(\mathcal{X}\times\mathcal{Y}\right)$ denotes the set of probability measures defined on the measurable space $(\mathcal{X}\times\mathcal{Y},\mathcal{B}(\mathcal{X}\times\mathcal{Y}))$, and $\mathcal{B}(\mathcal{X}\times\mathcal{Y})$ is the Borel sigma-algebra of $\mathcal{X}\times\mathcal{Y}$. Under Assumption~\ref{Assump -Primities OTP}, the collection of sets $\mathcal{B}_X=\left\{A\in\mathcal{B}(\mathcal{X}):A=\times_{i=1}^{d_x}(-\infty,x_i],x\in\mathcal{X}\right\}$ and $\mathcal{B}_Y=\left\{A\in\mathcal{B}(\mathcal{Y}):A=\times_{i=1}^{d_y}(-\infty,y_i],y\in\mathcal{Y}\right\}$ generate the sigma-algebras $\mathcal{B}(\mathcal{X})$ and $\mathcal{B}(\mathcal{Y})$, respectively. Consequently, there is no loss of information in reformulating the constraints in~(\ref{eq - constraint set Borel}) by replacing $\mathcal{B}(\mathcal{X})$ and $\mathcal{B}(\mathcal{Y})$ with $\mathcal{B}_X$ and $\mathcal{B}_Y$, respectively.

\par Now this reformulation of the constraint set is connected to moment inequalities using the cumulative distribution functions of $P_X$ and $P_Y$,
\begin{align*}
F_{X}(x)\coloneqq P_X\left(\times_{i=1}^{d_x}(-\infty,x_i]\right) \quad \text{and}\quad F_Y(y)\coloneqq P_Y\left(\times_{i=1}^{d_y}(-\infty,y_i]\right),
\end{align*}
which are defined on $\mathcal{X}$ and $\mathcal{Y}$, respectively. To clarify, I can express the $I$-projection problem (\ref{eq - Static Schrodinger prob}) as the following infinite-dimensional minimization problem,
\begin{equation}\label{eq - KL Min Problem}
\begin{aligned}
\text{minimize}& \quad m(p)\coloneqq\begin{cases} \int_{\Omega}p\log(p)\,dR_\gamma &\mbox{if } p\geq0,\quad \int_{\Omega}p\,dR_\gamma=1\\
\infty &\mbox{elsewhere},\end{cases} \\
\text{subject to}&\quad \int_{\Omega}g_{x^\prime}\, p\,dR_\gamma\leq 0\;\text{and}\;\int_{\Omega}-g_{x^\prime}\, p\,dR_\gamma\leq 0\quad\forall x^\prime\in\mathcal{X},\\
& \quad \int_{\Omega}g_{y^\prime}\, p\,dR_\gamma\leq 0\;\text{and}\;\int_{\Omega}-g_{y^\prime}\, p\,dR_\gamma\leq 0\quad\forall y^\prime\in\mathcal{Y},\\
&\quad \text{and}\quad p\in L_1(R_\gamma),
\end{aligned}
\end{equation}
where $\Omega\coloneqq\mathcal{X}\times\mathcal{Y}$, the moment functions are given by
\begin{align}\label{eq - moment fxns}
g_{x^\prime}(x,y)=F_{X}(x^\prime)-1\left[x\preceq x^\prime\right]\quad\text{and}\quad g_{y^\prime}(x,y)=F_{Y}(y^\prime)-1\left[y\preceq y^\prime\right],
\end{align}
with the notation "$\preceq$" means $x_i\leq x_i^\prime$ and $y_j\leq y_j^\prime$ for all $i=1,\ldots, d_x$ and $j=1,\ldots, d_y$, respectively, and
\begin{align*}
L_1(R_\gamma)\coloneqq\left\{h:\Omega\rightarrow\mathbb{R}:\,\text{$h$ is measurable $\mathcal{B}(\mathcal{X}\times\mathcal{Y})/\mathcal{B}(\mathbb{R})$ and}\;\int_{\Omega}|h|\,dR_\gamma<\infty\right\}.
\end{align*}
 
\par Let $\mathcal{V}\coloneqq \left\{g_h,-g_h:h\in\mathcal{X}\cup\mathcal{Y}\right\}$. By Lemmas~\ref{Lemma - Precomp} and \ref{Lemma - Closed}, this subset of $L_1(R_\gamma)$ is precompact and closed in the norm topology. These properties of $\mathcal{V}$ are central to this section's results: existence and exponential family representation of the dual optimizers, and approximation. Towards that end, I formulate the $I$-projection's constraint set equivalently as
\begin{align}\label{eq - constraint set M}
\mathcal{M}=\left\{p\in L_{1}(R_\gamma):m(p)<\infty,\;\int_{\Omega}v p\,dR_\gamma\leq 0\;\forall v\in\mathcal{V} \right\}.
\end{align}
The positive conjugate cone of $\mathcal{M}$, using $L_0(R_\gamma)$ as the dual space, is thus defined as $$\mathcal{M}^{\oplus}=\left\{z\in L_0(R_\gamma): \int_{\Omega}z\,p\,dR_\gamma\geq0\;\forall p\in\mathcal{M}\right\}.$$ In light of the form of $\mathcal{M}^{\oplus}$, I consider the dual optimization problem on the following domain
\begin{align}\label{eq - domain}
\mathcal{D}\coloneqq\left\{z\in\mathcal{M}^{\oplus}: z\in\overline{\text{co}}(\mathcal{V})\cdot\alpha,\;\alpha\geq0\right\},
 \end{align}
where $\overline{\text{co}}(\mathcal{V})$ is the closed convex hull of $\mathcal{V}$ in the $L_1(R_\gamma)$-norm. In particular, the dual optimization problem is given by
\begin{align}\label{eq - Fenchel Dual Problem}
\inf\left\{\int_{\Omega}e^{z}\,dR_\gamma;\; z\in\mathcal{D}\right\}.
\end{align}

\par I have the following result on the existence and uniqueness of the solution to the dual problem~(\ref{eq - Fenchel Dual Problem}), and its exponential representation. 
\begin{proposition}\label{Thm - Fenchel}
Let the constraint set $\mathcal{M}$ be given by~(\ref{eq - constraint set M}), and suppose that Assumption~\ref{Assump -Primities OTP} holds.
\begin{enumerate}
\item $\arginf\left\{\int_{\Omega}e^{z}\,dR_{\gamma}: z\in \mathcal{D}\right\}\neq\emptyset$ and the $R_\gamma$-density of $P_\gamma$ has the following representation
\begin{align}\label{eq - Representation}
p_\gamma=\frac{e^{z_{0,\gamma}}}{\int_{\Omega}e^{z_{0,\gamma}}\,dR_\gamma}\quad \text{where}\quad z_{0,\gamma}\equiv\arginf\left\{\int_{\Omega}e^{z}\,dR_\gamma: z\in \mathcal{D}\right\}.
\end{align}
solves the $I$-projection problem~(\ref{eq - KL Min Problem}). 
\item $z_{0,\gamma}\in\overline{\text{span}_{+}(\mathcal{V})}$, where $\overline{\text{span}_{+}(\mathcal{V})}$ is the $L_1(R_\gamma)$-norm closure of the positive linear span of
$\mathcal{V}$.
\end{enumerate}
\end{proposition}
\begin{proof}
See Appendix~\ref{Proof- Fenchel}.
\end{proof}
\noindent Furthermore, duality holds: 
\begin{align}\label{eq - strong duality}
m(p_\gamma)=-\log\left(\int_{\Omega}e^{z_{0,\gamma}}\,dR_\gamma\right).
\end{align}

\par I will now discuss the sieve approximation of the dual problem~(\ref{eq - Fenchel Dual Problem}). Observe that the domain is infinite-dimensional because there are infinitely many inequality restrictions. Therefore, I approximate the dual problem and implement it numerically using a simulation-based procedure. The sieve approximation scheme is the one put forward by~\cite{Tabri-MOOR-2025}. It approximates the dual problem using a sequence of finite-dimensional convex stochastic programs, and I use the SAA method in Section~\ref{Section Sieve M-Estimation} to approximate these stochastic programs.

\par I begin by reparametrizing the dual problem~(\ref{eq - Fenchel Dual Problem}) using the Gelfand–Pettis (G–F) integral. Because I work in $L_1(R_\gamma)$, I provide a self-contained definition and the minimal existence statement needed for this paper. This integral applies to \emph{any} subset $J\subset L_1(R_\gamma)$ that is compact in the norm topology. Suppose $\xi$ is a positive Radon measure on the measure space $(J,\mathcal{B}(J))$, where $\mathcal{B}(J)$ is the Borel sigma-algebra of $J$, and denote the topological dual of $L_1(R_\gamma)$ by $L_1(R_\gamma)^*$.
\begin{definition}[Gelfand--Pettis integral in $L_1(R_\gamma)$]\label{Def - GF integral}
 Suppose that for every continuous linear functional $\Lambda\in L_1(R_\gamma)^*$, the scalar map
$j\mapsto \Lambda(j)$ is $\xi$-integrable. If there exists $h\in L_1(R_\gamma)$
such that
\[
\Lambda(h) \;=\; \int_J \Lambda(j)\,d\xi(j)
\quad\text{for all }\Lambda\in L_1(R_\gamma)^*,
\]
then $h$ is called the \emph{Gelfand--Pettis integral} of the identity map on
$J$ with respect to $\xi$, and we write
\[
\int_J j\,d\xi(j) \;:=\; h.
\]
\end{definition}
\noindent As $L_1(R_\gamma)^*$ separates points on $L_1(R_\gamma)$, there is at most one such $h$ that satisfies Definition~\ref{Def - GF integral}. Thus, there is no uniqueness problem. The existence of $h$ follows from an application of Theorem~3.27 of~\cite{Rudin-Book}, because (a) $J$ is compact in the $L_1(R_\gamma)$-norm, and (b) $L_1(R_\gamma)$ with its norm topology is a Fr\'{e}chet space. As $\mathcal{V}$ is compact in the norm topology of $L_1(R_\gamma)$, I can apply this definition with the choice $J=\mathcal{V}$.

\par In consequence,  the dual problem~(\ref{eq - Fenchel Dual Problem}) can be reparametrized in terms of Radon measures whose supported on $\mathcal{V}$. These measures are modeled as elements of the space $C\left(\mathcal{V}\right)^{*}$--the topological dual of the Banach space of continuous functions on $\mathcal{V}$, $C(\mathcal{V})$. The reparametrization arises from the following representation of elements in $\mathcal{D}$: 
\begin{align}\label{eq - Result of Rudin Extended}
z\in\mathcal{D}\iff \exists\mu\in\mathcal{P}\,\text{and}\,\alpha\geq0\;\text{such that}\;z=\alpha\,\int_{\mathcal{V}}v\,d\mu(v).
\end{align}
where $\mathcal{P}\subset C\left(\mathcal{V}\right)^{*}$ the set of Radon probability measures on $\mathcal{V}$, and $\int_{\mathcal{V}}v\,d\mu(v)$ the G-F integral. Now define $\Xi\subset C\left(\mathcal{V}\right)^{*}$ as the set of all positive Radon measures on $\mathcal{V}$, and consider the following set
\begin{align}\label{eq - Upsilon set}
\Upsilon\coloneqq\left\{\xi\in\Xi: \xi=\alpha\cdot\mu,\alpha\geq0\,\text{and}\,\mu\in\mathcal{P}\right\}.
\end{align}
The dual problem~(\ref{eq - Fenchel Dual Problem}) and its solution set can now be reparametrized as
\begin{align}\label{eq - Reparametrized Fenchel Dual Problem}
\vartheta^*\coloneqq\inf\left\{\int_{\Omega}e^{\int_{\mathcal{V}}v\,d\xi(v)}\,dR_\gamma:\xi\in\Upsilon\right\},
\end{align}
and $\mathcal{S}^*\coloneqq\arginf\left\{\int_{\Omega}e^{\int_{\mathcal{V}}v\,d\xi(v)}\,dR_\gamma:\xi\in\Upsilon\right\}$, respectively. I have the following result on the characterization of $\mathcal{S}^*$.
\begin{corollary}\label{Coro - dual GF}
Suppose the conditions of Proposition~\ref{Thm - Fenchel} hold. Then it follows that $$\mathcal{S}^*=\left\{\xi\in\Upsilon: z_{0,\gamma}=\int_{\mathcal{V}}v\,d\xi(v) \right\}\neq\emptyset.$$
\end{corollary}
\begin{proof}
See Appendix~\ref{Proof- Fenchel G-F}.
\end{proof}
\noindent The G-F integral representation of the solution $z_{0,\gamma}$ in~(\ref{eq - Fenchel Dual Problem}) is $z_{0,\gamma}=\alpha_{0,\gamma}\,\int_{\mathcal{V}}v\,d\mu_{0,\gamma}(v)$. Since $z_{0,\gamma}$ is unique (up to equivalence class), it follows that $\alpha_{0,\gamma}$ is also unique, but that $\mu_{0,\gamma}$ may not be uniquely defined. This representation of $z_{0,\gamma}$ shows $\alpha_{0,\gamma}$ is also the total variation norm of the Radon measure $\xi_{0,\gamma}\coloneqq\alpha_{0,\gamma}\cdot\mu_{0,\gamma}$.

\par Using Assumption~\ref{Assump -Primities OTP}, I can also obtain a representation of $z_{0,\gamma}$ in terms of Schr\"odinger potentials (\citealp{nutz2021entropic}), enabling me to deduce an upper bound on $\alpha_{0,\gamma}/\gamma$ via an application of Lemma~4.9 in~\cite{nutz2022eot}.  The next result presents this bound.
\begin{proposition}\label{prop - alpha UB}
Suppose Assumption~\ref{Assump -Primities OTP} holds, and let $\|c\|_\infty=\sup_{(x,y)\in\mathcal{X}\times\mathcal{Y}}c(x,y)$. Then, $\alpha_{0,\gamma}/\gamma\leq \kappa \|c\|_\infty<\infty$, holds, where
$\kappa\coloneqq\frac{1}{\max\left\{1-F_X(\inf\mathcal{X}),1-F_Y(\inf\mathcal{Y})\right\}}$.
\end{proposition}
\begin{proof}
See Appendix~\ref{Proof- Prop prop - alpha UB}.
\end{proof}

\par The sieve I consider is based on the approximation scheme put forward by~\cite{Tabri-MOOR-2025}. His scheme is indexed by a sequence $\{\epsilon_\ell\}_{\ell\geq1}\subset\mathbb{R}_{++}$ with $\epsilon_\ell\downarrow0$ as $\ell\rightarrow\infty$. For each $\ell\in\mathbb{Z}_+$, let $U_\ell\coloneqq\left\{z\in L_1(R_\gamma): \|z\|_{L_1(R_\gamma)}\leq\epsilon_\ell\right\}.$ Then by Lemma~A.2 of~\cite{Tabri-MOOR-2025}, there corresponds a finite partition $\{E_{i,\ell}\}_{i=1}^{n_\ell}$ of $\mathcal{V}$ with the property
\begin{align}\label{eq - Disc Accuracy}
\int_{\mathcal{V}}v\,d\mu(v)-\sum_{i=1}^{n_\ell}\mu(E_{i,\ell})v_i\in U_\ell\quad \forall v_i\in E_{i,\ell},\;i=1\,\ldots,n_\ell.
\end{align}
Remarkably, the only structure on the partitions that is required for this approximation to hold is that for each $i$: $v-v^\prime\in U_\ell$ for all $v,v^\prime\in E_{i,\ell}$. Consequently, the partitions depend only on $\epsilon_\ell$, $R_\gamma$ and $\mathcal{V}$, and not $\mu$. Of course, to satisfy the accuracy~(\ref{eq - Disc Accuracy}) the dimension of the discretization, $n_\ell$, will diverge to $\infty$ as $\ell\rightarrow\infty$. 

\par There are many partitions of $\mathcal{V}$ that can be used in practice. By the proof of Lemma~\ref{Lemma - Precomp}, the set $\mathcal{V}$ is VC subgraph; hence, there exists a finite partition of $\mathcal{V}$ satisfying the accuracy~(\ref{eq - Disc Accuracy}) that does \emph{not} depend on $R_\gamma$. Partitions of this sort are advantageous in practice since $n_\ell$ would then be independent of the regularizer $\gamma$. In particular, standard entropy bounds for uniformly bounded 
VC-subgraph classes (e.g., Theorem 2.6.7 in~\citealp{VDV-W}) imply that, for each approximation level $\epsilon_\ell$, 
one can choose a partition such that $n_\ell$ and the accuracy~(\ref{eq - Disc Accuracy}) satisfy $n_\ell \;\le\; K\,\epsilon_\ell^{-\delta}$,
for constants $K,\delta>0$ depending only on the VC characteristics of $\mathcal{V}$ and the uniform envelope (here, it is the constant function equal to 1), but not on the reference measure $R_\gamma$. Consequently, $n_\ell$ can always be taken to grow at most polynomially in $\epsilon_\ell^{-1}$, uniformly over the choice of the regularization parameter 
$\gamma$. Appendix~\ref{Section - Partition of V} presents an example of such a partition.  

\par The sieve approximation for the dual problem proceeds as follows. 
Fix a level $\ell\in\mathbb{Z}_{+}$ and choose an approximation tolerance 
$\epsilon_\ell>0$. Let $n_\ell\in\mathbb{Z}_{+}$ and 
$\{E_{i,\ell}\}_{i=1}^{n_\ell}$ be a partition of $\mathcal{V}$ 
satisfying the accuracy requirement~\eqref{eq - Disc Accuracy}. 
For each $i=1,\ldots,n_\ell$, select a representative $v_i\in E_{i,\ell}$ and 
let
\begin{align}\label{eq - moment function SAA}
G_\ell(\omega;\alpha,\mu)
:=\exp\!\left\{\alpha\sum_{i=1}^{n_\ell}\mu_i\,v_i(\omega)\right\},
\qquad
\omega\in\Omega,\; (\alpha,\mu)\in\mathcal{C}_\ell,
\end{align}
where $\mathcal{C}_\ell \coloneqq\left\{(\alpha,\mu)\in\mathbb{R}_+^{n_\ell+1}:\sum_{i=1}^{n_\ell}\mu_i=1\right\}$.
The finite-dimensional convex program associated with level~$\ell$ is
\begin{align}
\vartheta_\ell^*
&:=\inf_{(\alpha,\mu)\in\mathcal{C}_\ell}
\int_{\Omega} G_\ell(\omega;\alpha,\mu)\,dR_\gamma(\omega), 
\label{eq - finite program}
\\
\mathcal{S}_\ell^*
&:=\arginf_{(\alpha,\mu)\in\mathcal{C}_\ell}
\int_{\Omega} G_\ell(\omega;\alpha,\mu)\,dR_\gamma(\omega).
\label{eq:finite-program-solutions}
\end{align}

\noindent I can also apply Theorem~3 of~\cite{Tabri-MOOR-2025} to establish convergence of the approximation scheme to the original problem~(\ref{eq - Reparametrized Fenchel Dual Problem}) with the sequence of approximate solutions converging to a solution of the original problem, as $\ell\rightarrow \infty$. The next result formalizes this point. 
\begin{proposition}\label{Thm - Computation}
Suppose Assumption~\ref{Assump -Primities OTP} holds. Let $\bar{c}=\int_\Omega c(\omega)\,\,d(P_X\otimes P_Y)$ and let $\{\epsilon_\ell,U_\ell\}_{\ell\geq1}$ be described as above. For each $\ell$, let $\{E_{i,\ell}\}_{i=1}^{n_\ell}$ be a partition of $\mathcal{V}$ such that, for each $i$, $v-v^\prime\in U_\ell$ holds for all $v,v^\prime\in E_{i,\ell}$. Then, for each $\ell$,
$\mathcal{S}^*_\ell\neq\emptyset$
for any $v_i\in E_{i,\ell}$ where $i=1,\ldots,n_\ell$. Furthermore, for each $\ell$ and $v_i\in E_{i,\ell}$ with $i=1,\ldots,n_\ell$, define the corresponding Radon measure $\xi_\ell=\alpha_{n_\ell}\sum_{i=1}^{n_\ell}\mu_{i,n_\ell} \delta_{v_i}$, where $(\alpha_{n_\ell},\mu_{1,n_\ell},\ldots,\mu_{n_\ell,n_\ell})\in\mathcal{S}^*_\ell$ and $\delta_{v_i}$ is the Dirac delta function at $v_i$ for each $i$. Then the following statements hold.
\begin{enumerate}
\item $\lim_{\ell\rightarrow\infty}\vartheta^*_\ell=\vartheta^*$,
\item Convergence rate of $\{\log(\vartheta^*_\ell)/\gamma\}_{\ell\geq 1}$: $\gamma^{-1}\left|\log\vartheta^*_\ell-\log\vartheta^*\right|\leq \epsilon_\ell\,\kappa\|c\|_{\infty}e^{\gamma(\kappa\|c\|_{\infty}+\bar{c})+\log a_\gamma}$ for all $\ell\in\mathbb{Z}_+$. 
\item Every accumulation point of $\left\{\xi_\ell\right\}_{\ell\geq1}$, in the weak-star topology of $C\left(\mathcal{V}\right)^{*}$, is an element of $\mathcal{S}^*$.
\end{enumerate}
\end{proposition}
\begin{proof}
See Appendix~\ref{Proof Prop computation}.
\end{proof}

\par This section achieved two complementary objectives. First, it reformulated the
entropically regularized optimal transport problem as an  $I$-projection
under an infinite collection of moment inequality restrictions and characterized
its solution through a Fenchel dual representation involving a
G-F integral. Second, it introduced a deterministic sieve
approximation of this dual problem, replacing the infinite-dimensional
optimization with a sequence of finite-dimensional convex programs whose
solutions converge to those of the original dual problem as the sieve resolution
increases. These approximating programs provide a numerically tractable
representation of the dual problem, but they still involve expectations with
respect to the reference measure and are therefore not directly observable.

The following section builds on this deterministic approximation by introducing a
sieve M-estimation framework. Specifically, the finite-dimensional convex
programs constructed above are estimated using the SAA
method, yielding estimators of both the dual optimal value and the
corresponding dual variables. This two-step structure--deterministic sieve
approximation followed by stochastic estimation--forms the basis for the
consistency, rate, and inference results established in the remainder of the
paper.

\section{Sieve M-Estimation}\label{Section Sieve M-Estimation}
\par As the approximating problem~(\ref{eq - finite program}) is a stochastic program, I describe a procedure for computing $\vartheta^*_\ell$ and elements of $\mathcal{S}^*_\ell$ using the SAA method. I also show the procedure converges to the Fenchel dual problem~(\ref{eq - Reparametrized Fenchel Dual Problem}). Firstly, replace the constraint set $\mathcal{C}_\ell$ with $\overline{\mathcal{C}}_\ell\coloneqq\left\{(\alpha,\mu)\in\mathcal{C}_\ell:\alpha\leq\gamma\kappa\|c\|_\infty\right\}$, where the upper bound is due to Proposition~\ref{prop - alpha UB}. Consequently, for each $\ell$, the collection of functions
\begin{align}\label{eq - Collection VC}
\mathcal{G}_{\ell}\coloneqq\left\{G_\ell(\cdot,\alpha,\mu):(\alpha,\mu)\in\overline{\mathcal{C}}_\ell\right\}
\end{align}
is a uniformly bounded class of functions, with common bound $e^{\gamma\kappa\|c\|_\infty}$. This property of $\mathcal{G}_{\ell}$ is useful for the ensuing analysis, as the level of the discretization's accuracy, $\epsilon_\ell$, must be coupled with the sample size through the complexity of this class of functions. 

\par Next, replace the finite program~(\ref{eq - finite program}) and its solution set $\mathcal{S}^*_\ell$ with
\begin{align}
\vartheta_\ell\coloneqq\inf\left\{\int_{\Omega}G_\ell(\omega;\alpha,\mu)\,dR_\gamma(\omega): (\alpha,\mu)\in\overline{\mathcal{C}}_\ell\right\}\quad\text{and}\label{eq - finite program 1}\\
S_\ell\coloneqq\argmin\left\{\int_{\Omega}G_\ell(\omega;\alpha,\mu)\,dR_\gamma(\omega): (\alpha,\mu)\in\overline{\mathcal{C}}_\ell\right\},\nonumber
\end{align}
respectively. Now solving this optimization problem using the SAA method entails solving the sample-analogue of~(\ref{eq - finite program 1}) with a simulated random sample from $R_\gamma$. Let $\omega^{N_\ell}\coloneqq\{\omega_j,i\leq N_\ell\}$ be a random sample of size $N_\ell$ from $R_\gamma$. The SAA method solves
 \begin{align}\label{eq - SAA 0}
\hat{\vartheta}_\ell\left(\omega^{N_\ell}\right)\coloneqq\inf\left\{\frac{1}{N_\ell}\sum_{j=1}^{N_\ell}G_\ell(\omega_j,\alpha,\mu):(\alpha,\mu)\in\overline{\mathcal{C}}_\ell\right\}
\end{align}
and approximates $S_\ell$ with $\hat{S}_\ell\left(\omega^{N_\ell}\right)\coloneqq\argmin\left\{\frac{1}{N_\ell}\sum_{j=1}^{N_\ell}G_\ell(\omega_j,\alpha,\mu): (\alpha,\mu)\in\overline{\mathcal{C}}_\ell\right\}$. 

\begin{remark}[Measurability]\label{Remark - Measurability SAA}
For each $\ell\in\mathbb{Z}_+$, the SAA objective function in~(\ref{eq - finite program 1}) is defined on a common probability space $\left(\times_{i=1}^{N_\ell}\Omega, \times_{i=1}^{N_\ell}\mathcal{B}(\Omega), R^{\otimes_{N_\ell}}_\gamma\right)$, where $\mathcal{B}(\Omega)$ is the Borel sigma-algebra of $\Omega$, and $R^{\otimes_{N_\ell}}_\gamma)$ is the $N_\ell$-fold product of the probability measure $R_\gamma$. As the function $G_\ell(\omega;\alpha,\mu)$, defined in~(\ref{eq - moment function SAA}), is a \emph{Carath\'{e}odory function}, i.e., continuous in $(\alpha,\mu)$ and measurable in $\omega$, the SAA objective function is also such a function. In consequence, $\hat{\vartheta}_\ell\left(\omega^N\right)$ and $\hat{S}_\ell\left(\omega^N\right)$ are measurable. Furthermore, a particular optimal solution $(\alpha_\ell,\mu_\ell)$ of the SAA problem is a measurable selection $\left(\alpha_\ell\left(\omega^{N_\ell}\right),\mu_\ell\left(\omega^{N_\ell}\right)\right)\in\hat{S}_{\ell}\left(\omega^{N_\ell}\right)$, and existence of such measurable selection is ensured by the Measurable Selection Theory (e.g., Theorem~7.34,~\citealp{Shapiro-Dentcheva-Ruszczynski}). This takes care of measurability questions.
\end{remark}

\par In my setup, the sequence of SAA programs~(\ref{eq - SAA 0}) comprise a nonparametric sieve estimation problem. A sufficient condition for its consistency, as the sample size $N_\ell$ increases with $\ell$, for fixed $\gamma$, is that
\begin{align}\label{eq - entropy bound}
\lim_{\ell\rightarrow\infty}N^{-1}_\ell\log n_\ell=0.
\end{align}
This condition implies the Rademacher complexity, $\mathcal{R}_{N_\ell}\left(\mathcal{G}_\ell\right)$, described in Lemma~\ref{Lemma - Rademacher complexity G} and Corollary~\ref{Corollary - Rademacher complexity}, satisfies $\mathcal{R}_{N_\ell}\left(\mathcal{G}_\ell\right)=o(1)$ as $\ell\rightarrow\infty$. Consequently, I can establish that a uniform strong law of large numbers holds for the sequence of function classes $\{\mathcal{G}_\ell\}_{\ell\geq1}$--see Proposition~\ref{Prop - ULLN} in Appendix~\ref{Appendix - ULLN}. This large-sample result is key in the proof of my main result.
\begin{theorem}\label{Thm Stochastic Opt Sieve}
Suppose that Assumption~\ref{Assump -Primities OTP} holds, let $\mathcal{G}_{\ell}$ be as given in~(\ref{eq - Collection VC}) for each $\ell\in\mathbb{Z}_+$, and suppose that the limit~(\ref{eq - entropy bound}) holds. For each $\ell$, let $\omega^{N_\ell}$ be a random sample from $R_\gamma$, and $\hat{\xi}_\ell\left(\omega^{N_\ell}\right)=\hat{\alpha}_\ell\left(\omega^{N_\ell}\right)\sum_{i=1}^{n_\ell}\hat{\mu}_{i,\ell}\left(\omega^{N_\ell}\right) \delta_{v_i}$ where $\left(\hat{\alpha}_\ell\left(\omega^{N_\ell}\right),\hat{\mu}_\ell\left(\omega^{N_\ell}\right)\right)\in\hat{S}_\ell\left(\omega^{N_\ell}\right)$. Then the following statements hold. 
\begin{enumerate}
\item With probability 1,
\begin{enumerate}[(i)]
\item $\lim_{\ell\rightarrow\infty}\hat{\vartheta}_\ell=\vartheta^*$.
\item Every accumulation point of $\{\hat{\xi}_\ell\}_{\ell\geq1}$, in the weak-star topology of $C\left(\overline{\mathcal{V}}\right)^{*}$, is an element of $\mathcal{S}^*$. 
\end{enumerate}
\item Sample complexity: for each $\ell$, $$\gamma^{-1}E_{R^{\otimes_{N_\ell}}_\gamma}\left[\left|\log\hat{\vartheta}_\ell -\log\vartheta^*\right|\right]\leq\max\left\{\sqrt{\frac{2\log n_\ell}{N_\ell}},\epsilon_\ell\,\right\}2\kappa\|c\|_{\infty}e^{\gamma2\kappa\|c\|_{\infty}}.$$

\item As $\ell\to\infty$,
\begin{align*}
\vartheta^*
&\;\ge\;
\hat{\vartheta}_\ell
- N_\ell^{-1/2}\sup_{g\in\mathcal{G}_\ell} B^{l}_{\ell}(g)
- \epsilon_\ell\, \gamma\kappa \|c\|_{\infty}e^{\gamma\kappa \|c\|_{\infty}}
+ o_{R_\gamma}\!\bigl(N_\ell^{-1/2}\bigr),\\
\vartheta^*
&\;\le\;
\hat{\vartheta}_\ell
+ N_\ell^{-1/2}\sup_{g\in\mathcal{G}^{-}_\ell} B^{u}_{\ell}(g)
+ o_{R_\gamma}\!\bigl(N_\ell^{-1/2}\bigr),
\end{align*}
where $B^{l}_{\ell}$ and $B^{u}_{\ell}$ are centered Gaussian processes
indexed by $\mathcal{G}_\ell$ and $\mathcal{G}_\ell^{-}=-\mathcal{G}_\ell$
with common covariance kernel
$\int_{\Omega} g(\omega)g'(\omega)\, dR_\gamma(\omega)$.
\end{enumerate}
\end{theorem}
\begin{proof}
See Appendix~\ref{Proof - Thm Stoch Opt Sieve}.
\end{proof}
\noindent An immediate consequence of the result in Part 1(i) of Theorem~\ref{Thm Stochastic Opt Sieve} is on estimation of the EOT value~(\ref{eq - Entropc OTP}). In the notation of Section~\ref{Section Framing as MIs}, this value is 
given by $\gamma^{-1}\left(\log a_\gamma+m(p_\gamma)\right)$, and with probability 1,
\begin{align*} 
\lim_{\ell\rightarrow\infty}\gamma^{-1}\left(\log a_\gamma-\log\hat{\vartheta}_\ell\right)=\gamma^{-1}\left(\log a_\gamma-\log\vartheta^*\right)=\gamma^{-1}\left(\log a_\gamma+m(p_\gamma)\right),
\end{align*}
by an application of the Continuous Mapping Theorem and optimal value duality~(\ref{eq - strong duality}).

\par
Sections~\ref{Section Framing as MIs}--\ref{Section Sieve M-Estimation} together
yield a two-stage construction: a deterministic sieve approximation of the
Fenchel dual of the $I$-projection problem~\eqref{eq - Fenchel Dual Problem},
followed by a stochastic approximation of the resulting finite-dimensional
convex programs via the SAA method. Theorem~\ref{Thm Stochastic Opt Sieve} summarizes the statistical consequences
of this construction, establishing (i) almost sure consistency for the
estimators of the dual optimal value and the associated dual solutions,
(ii) finite-sample mean-error bounds for the estimator of the primal optimal
value, and (iii) stochastic upper and lower error bounds for the estimation
error of the dual optimal value. The remainder of the paper discusses how these guarantees compare with the
existing statistical optimal transport literature. 

\section{Discussion}\label{Section Discussion}
This section contrasts the present moment-inequality-based approach with the
empirical Sinkhorn divergence and with continuity-driven dual parametrizations,
and it explains how the population-level constraint structure underlying the
sieve M-estimator leads to different modes of convergence and different
inferential tools. It also outlines extensions to entropic optimal transport
problems with additional moment restrictions.
\subsection{Other Approaches}

\par A large body of work on approximating entropically regularized optimal transport
focuses on computational formulations of the Fenchel--Rockafellar dual, often
implemented via the Sinkhorn algorithm. In the general setting, the Fenchel--Rockafellar
dual can be written as in \cite{nutz2021entropic},
\begin{align*}
\sup_{f\in L_1(P_X),\, g\in L_1(P_Y)}
& \left(
\int_{\mathcal{X}} f\,dP_X
+
\int_{\mathcal{Y}} g\,dP_Y
\right.\\
&\qquad\left.-
\frac{1}{\gamma}
\Big(
\int_{\mathcal{X}\times\mathcal{Y}}
e^{\gamma(f(x)+g(y)-c(x,y))}
\, d(P_X\otimes P_Y)
- 1
\Big)
\right),
\end{align*}
\noindent where the functions $f$ and $g$ are called dual or Schr\"odinger potentials in the literature. Most practical algorithms do not use this $L_1$-based convex-analytic dual. Instead, they require the potentials to be continuous and use parametrizations ensuring pointwise evaluation and stochastic-gradient updates are well defined. This category includes the RKHS-based sieve of \cite{Genevey} and the neural-network parametrizations of \cite{Seguy} and \cite{Arjovsky2017WGAN}. The RKHS approach assumes the true potentials belong to a fixed RKHS—a strong modeling assumption that fails for many continuous functions on compact sets, as demonstrated by Theorem~1.1 of \cite{Steinwart2024}. Neural-network parametrizations avoid this restriction but still require continuity of the potentials as a structural assumption motivated by computational considerations.

The dual framework developed in this paper fundamentally departs from continuity-based approaches. Here, the dual variables are finite Radon measures that act on a VC class of moment functions, placing them naturally in the dual space $C(\mathcal{V})^*$ with the weak-star topology. Unlike previous methods, this framework does not require continuity of the Schr\"odinger potentials. This choice aligns with the analytical structure of the entropic dual: regularization removes the pointwise constraint $f(x) + g(y) \le c(x, y)$ found in unregularized optimal transport, so integrability—not continuity—is the minimal regularity ensured by the Schr\"odinger system. As a result, the sieve M-estimator developed here preserves the dual problem's population-level structure without imposing unnecessary smoothness conditions.

A comparison with the empirical Sinkhorn divergence further clarifies the fundamental differences in how these estimators use data and the EOT constraint set. The sieve M-estimator enforces the EOT constraints at the population level by approximating the moment inequalities that define the Schr\"odinger system, guided by the geometry of the population-level feasible set of couplings. In contrast, the empirical Sinkhorn divergence replaces the true marginals $P_X$ and $P_Y$ with their empirical versions, solving a discrete EOT problem that enforces only empirical marginal constraints and derives its geometry entirely from the sampled data. As a result, the sieve M-estimator preserves the analytical structure of the population constraints and uses data solely to approximate expectations in the dual objective. In contrast, the empirical Sinkhorn divergence treats the empirical distributions as if they were the population distributions. This distinction has practical implications: the two estimators respond differently to sampling variability and to the geometry of the underlying distributions, and only the sieve estimator maintains the population constraint set inherent in the EOT problem.

These structural differences shape their statistical guarantees. Notably, only the empirical Sinkhorn divergence currently has established convergence rates or central limit theorems—and even these results hold only in special cases involving smooth or quadratic costs. In contrast, Theorem~\ref{Thm Stochastic Opt Sieve} shows that the sieve M-estimator achieves consistency under much weaker conditions and provides finite-sample mean error bounds and stochastic approximations valid for all continuous cost functions on compact supports. I discuss these results in detail below.

\subsection{Part 1 of Theorem~\ref{Thm Stochastic Opt Sieve}}

The first part of Theorem~\ref{Thm Stochastic Opt Sieve} establishes the almost sure convergence of both the sieve optimal value estimator and the corresponding dual estimators, for any continuous cost function on compact supports. This significantly broadens the scope of statistical optimal transport beyond the traditional focus on smooth or quadratic costs and empirical Sinkhorn divergence. The main distinction between my dual consistency result and previous empirical Sinkhorn analyses lies in the mode and topology of convergence: my dual estimators converge almost surely in the weak-star topology of $C(\mathcal{V})^*$, whereas empirical Sinkhorn results typically establish convergence in mean with respect to the $L_2(\hat{P}_X \otimes \hat{P}_Y)$ norm, where $\hat{P}_X$ and $\hat{P}_Y$ are the empirical marginals (see, e.g., Lemma~4.8 in Section~4.2 of \citealp{ChewiNilesWeedRigollet2025}). The question of deriving convergence rates for my dual variable estimator remains open and is left for future work, though I briefly outline potential approaches below.

\par This difference in convergence mode has important practical implications. Weak-star almost sure convergence offers a strong qualitative guarantee: the estimated dual variables converge in the weak-star topology for almost every sample path, not just on average. This property is especially valuable for downstream tasks--such as sensitivity analysis, specification testing, or constructing confidence bounds for transport costs--where a stable dual representation is crucial. In contrast, mean $L_2$ convergence for empirical Sinkhorn duals depends on the empirical marginals and reflects approximation quality averaged over the sampled supports. While both types of convergence are informative, weak-star almost sure convergence uniquely captures how my dual variables utilize population-level information, rather than only the empirical support. This makes it particularly well suited for applications that rely on the dual structure inherent to the EOT problem.

\par A natural starting point for studying rates of convergence of my dual variable
estimator is to endow the parameter space with a compatible metric. Recall 
that the dual variables lie in the set $\Upsilon_0 
  \;=\; 
  \bigl\{
    \xi \in \Upsilon : \|\xi\|_{TV} \le \gamma\kappa\|c\|_\infty 
  \bigr\}$, where $\|\cdot\|_{TV}$ is the total variation norm on $C(\mathcal{V})^*$. As it is a norm-bounded subset of $C(\mathcal{V})^*$, by the 
Banach--Alaoglu theorem it must be weak-star compact, and the weak-star topology on it is metrizable. A convenient metric 
arises from the parametrization
\[
  \Phi : (0,\gamma\kappa\|c\|_\infty] \times \mathcal{P} \to \Upsilon_0,
  \qquad
  \Phi(\alpha,\mu) = \alpha \mu,
\]
where $\mathcal{P}$ denotes the set of Radon probability measures on 
\(\mathcal{V}\). Endow $(0,\gamma\kappa\|c\|_\infty]$ with the usual 
metric and $\mathcal{P}$ with any metric that metrizes weak convergence 
(e.g., the bounded--Lipschitz or Prokhorov metric). Then $\Phi$ is a 
homeomorphism between the product space and $(\Upsilon_0,\text{weak-star})$, 
so the weak-star topology on $\Upsilon_0$ may be described by a product metric. 
This makes it possible, at least in principle, to study rates of convergence of 
the dual estimators by analyzing the rates at which the scalar components 
$\hat{\alpha}_\ell$ and the probability-measure components $\hat{\mu}_\ell$ 
approach their population counterparts. A full development of such rates would 
require delicate stability properties of the dual problem and is therefore left 
for future research, but this parametrization clarifies how a metric-based 
analysis could proceed.

\subsection{Part 2 of Theorem~\ref{Thm Stochastic Opt Sieve}}
\subsubsection{Comparison to Theorem~3 in~\cite{GenevayCuturiPeyreBach2019}}

The second result of Theorem~\ref{Thm Stochastic Opt Sieve} provides a
finite-sample rate on the approximation error incurred by the sieve
M-estimator when estimating the optimal value of the EOT problem, namely
\[
\gamma^{-1}\bigl( \log a_\gamma + \log \vartheta^* \bigr),
\]
under Assumption~\ref{Assump -Primities OTP}. It is instructive to compare
this bound with the in-mean convergence rate of the empirical Sinkhorn
divergence established in Theorem~3 of~\cite{GenevayCuturiPeyreBach2019}.
Their result applies when the cost function is infinitely differentiable
and $L$-Lipschitz and when $\mathcal{X}$ and $\mathcal{Y}$ are bounded
subsets of $\mathbb{R}^d$. In our notation, their discrepancy is bounded
(up to constants) by
\[
\frac{e^{\gamma(2L|\mathcal{X}|+\|c\|_{\infty})}}{\sqrt{N}}
\Bigl(1 + \gamma^{\lfloor d/2\rfloor}\Bigr),
\]
where $|\mathcal{X}|$ denotes the diameter of $\mathcal{X}$, and the
constants depend only on $|\mathcal{X}|$, $|\mathcal{Y}|$, $d$, and
$\|c^{(k)}\|_{\infty}$ for $k = 0, \ldots, \lfloor d/2\rfloor$.

There are several points of comparison and contrast.

\textbf{(1) Scope of applicability.}
My rate applies to \emph{all continuous} cost functions on compact
supports, whereas theirs requires the substantially narrower class of
infinitely differentiable, globally Lipschitz costs. Both bounds exhibit
exponential dependence on~$\gamma$. In the quadratic case 
$c(x,y)=\|x-y\|^2/2$, \cite{MenaNilesWeed2019} show that this exponential
dependence can be removed for the empirical Sinkhorn divergence. It is
plausible that an analogous refinement may be possible for my sieve
M-estimator, and I leave this investigation for future research.

\textbf{(2) Dependence on dimension and regularization.}
In~\cite{GenevayCuturiPeyreBach2019}, the factor 
$\bigl(1 + \gamma^{\lfloor d/2\rfloor}\bigr)$ introduces a pronounced
curse of dimensionality, especially for large~$\gamma$. In contrast, my
rate depends on dimension and regularization only through $\log n_\ell$,
where $n_\ell$ is the size of the sieve. When the partition of
$\mathcal{V}$ is chosen independently of $R_\gamma$, this
removes the $\gamma^{d/2}$ effect entirely. Even when the sieve is
adapted to geometric or distributional features of $P_X$, $P_Y$, or
$R_\gamma$, the influence of such features is substantially attenuated by
the logarithm.

\textbf{(3) Nature of the convergence rate.}
The rate obtained for the sieve estimator is nonparametric--it depends on
$\max\left\{\epsilon_\ell,\ \sqrt{2\log(n_\ell)/N_\ell}\right\}$,
and therefore may be slower than $N^{-1/2}$. Nevertheless, the numerical
experiment in Appendix~\ref{Section - Implementation} shows that, for a moderate sample
size and a well-chosen sieve, the proposed estimator can outperform the
empirical Sinkhorn divergence in practice, despite its theoretically
slower rate. This highlights the practical advantage of adapting the
approximation to the structure of the $I$-projection rather than to the
sample-induced discrete geometry.

\subsubsection{Bias--variance trade-off in the rate}

\par The rate established in Part~2 of Theorem~\ref{Thm Stochastic Opt Sieve}
makes explicit the bias--variance trade-off inherent in the nonparametric
estimation of $\gamma^{-1}\log\vartheta^*$. As is typical in nonparametric
settings, the stochastic component of the error depends on the size and
complexity of the sieve function class $\mathcal{G}_\ell$. The term
$\sqrt{2\log(n_\ell)/N_\ell}$ increases with $n_\ell$ and, by
Lemma~S3.1, quantifies this complexity through a Rademacher-type bound.
Enlarging $\mathcal{G}_\ell$ introduces greater variability into
$\gamma^{-1}\log\hat{\vartheta}_\ell$, thereby inflating the variance,
whereas the deterministic approximation error $\epsilon_\ell$ represents a
bias term that decreases as $\mathcal{G}_\ell$ becomes richer. Thus,
reducing bias inevitably increases variance, reflecting the classical
bias--variance trade-off.

\par In standard nonparametric estimation, the sample size is fixed and the
optimal convergence rate is obtained by choosing a sieve dimension that
balances the bias and variance terms. By contrast, when the EOT problem~(\ref{eq - Entropc OTP}) is
viewed as a stochastic optimization problem, the present framework first
specifies the deterministic approximation error $\epsilon_\ell$, which in turn
determines the size of $\mathcal{G}_\ell$, and then selects a sample size that
satisfies the rate condition~\eqref{eq - entropy bound}. The optimal balance
is achieved by equating the stochastic and deterministic errors:
\begin{align}\label{eq - Optimal Sample size}
\sqrt{\frac{2\log n_\ell}{N_\ell}}=\epsilon_\ell
\quad\Longleftrightarrow\quad
N_\ell = \frac{2\log n_\ell}{\epsilon_\ell^2}.
\end{align}
A sample size of this order automatically satisfies the entropy
condition~\eqref{eq - entropy bound} and yields the corresponding optimal rate
of convergence for the sieve M-estimator.

\par Beyond clarifying the bias--variance trade-off, this calculation provides
practically useful guidance on sample size selection. To the best of my
knowledge, the Statistical Optimal Transport literature offers no explicit
prescriptions for determining the number of samples required to achieve a
desired accuracy level---whether for empirical Sinkhorn, RKHS-based methods,
or neural-dual parametrizations. The expression
$N_\ell = 2\,\epsilon_\ell^{-2}\log n_\ell$
gives a concrete, interpretable sample size that links the target
accuracy~$\epsilon_\ell$ to the complexity of the sieve. Such guidance is
especially valuable in applications where practitioners must tune both the
regularization parameter $\gamma$ and the resolution of the approximation
scheme, yet currently have no principled tool for determining how large a
sample is needed for a given precision.

\subsection{Part 3 of Theorem~\ref{Thm Stochastic Opt Sieve}}
\subsubsection{Comparison with Empirical Sinkhorn Divergence}
\par
It is instructive to contrast the stochastic bounds in Part~3 of
Theorem~\ref{Thm Stochastic Opt Sieve} with recent results on the sample
complexity of the empirical Sinkhorn divergence, such as Theorem~3
in~\cite{RigolletStromme2024SampleComplexitySinkhorn}.
The latter provides finite-sample concentration inequalities that control the
estimation error in probability at a fixed sample size, for EOT value~(\ref{eq - Entropc OTP}). By contrast, the present result yields an asymptotic stochastic approximation for
the sieve M-estimator of the optimal value of the Fenchel dual problem,
$\vartheta^*$. Here, the estimation error is bounded above and below by suprema of Gaussian
processes indexed by the sieve class and scaled by $N_\ell^{-1/2}$,
up to a negligible remainder in probability.

These two types of results address complementary inferential questions.
Concentration inequalities quantify tail behavior of the empirical Sinkhorn
divergence viewed as a discretized transport problem.
The Gaussian-process bounds derived here describe the full asymptotic
fluctuation envelope of a population-level stochastic optimization problem,
and are naturally suited to inference procedures such as confidence intervals
and projection-based uncertainty quantification.

\subsubsection{Confidence Interval for EOT value}
\par The third part of Theorem~\ref{Thm Stochastic Opt Sieve} provides matching
upper and lower stochastic inequalities for the estimation error of
$\hat{\vartheta}_\ell$. These bounds naturally lead to the construction of an
asymptotically valid confidence interval for $\vartheta^*$, say $[b_\ell,b^\prime_\ell]$.
Because the map $\vartheta \mapsto \gamma^{-1}(\log a_\gamma - \log \vartheta)$ is
strictly decreasing on $(0,\infty)$, this interval can be projected into an
asymptotically valid confidence interval for the regularized optimal transport
value $\gamma^{-1}(\log a_\gamma - \log \vartheta^*)$ via
\[
\Bigl[
\gamma^{-1}\bigl(\log a_\gamma - \log b^\prime_\ell\bigr),
\quad
\gamma^{-1}\bigl(\log a_\gamma - \log b_\ell\bigr)
\Bigr].
\]

\par Let $q^{(l)}_{\delta,\ell}$ and $q^{(u)}_{\delta,\ell}$ denote the
$\delta$-quantiles of $\sup_{g\in\mathcal{G}_\ell} B^{l}_{\ell}(g)$ and
$\sup_{g\in\mathcal{G}^{-}_\ell} B^{u}_{\ell}(g)$, respectively.
At the population level, Part~3 of Theorem~\ref{Thm Stochastic Opt Sieve}
suggests the asymptotic $(1-\delta)$ confidence interval
\[
\Bigl[
\hat{\vartheta}_\ell
- N_\ell^{-1/2} q^{(l)}_{1-\delta/2,\ell}
- \epsilon_\ell\, \gamma \kappa\|c\|_{\infty}e^{\gamma \kappa \|c\|_{\infty}},
\quad
\hat{\vartheta}_\ell
+ N_\ell^{-1/2} q^{(u)}_{1-\delta/2,\ell}
\Bigr]
\]
for $\vartheta^*$. Since $\mathcal{G}_\ell^-=-\mathcal{G}_\ell$, the Gaussian
processes $B^{l}_{\ell}$ and $B^{u}_{\ell}$ share the same covariance kernel and
therefore have the same distribution, albeit indexed by different classes.

\par The lower bound includes a deterministic bias term
$\epsilon_\ell\, \gamma \kappa\|c\|_{\infty}e^{\gamma \kappa \|c\|_{\infty}}$,
which may be large for moderate or large values of $\gamma$.
For practical implementation, it is therefore natural to consider the symmetric
interval
\begin{equation}\label{eq:SymmetricCI}
\Bigl[
\hat{\vartheta}_\ell
- N_\ell^{-1/2} q^{(l)}_{1-\delta/2,\ell},
\;
\hat{\vartheta}_\ell
+ N_\ell^{-1/2} q^{(u)}_{1-\delta/2,\ell}
\Bigr],
\end{equation}
which omits this bias term. This modification is asymptotically valid under a
standard undersmoothing condition: if $\sqrt{N_\ell}\epsilon_\ell=o(1)$, in
addition to the rate condition~\eqref{eq - entropy bound} for fixed $\gamma$,
then the omitted bias is negligible at the $N_\ell^{-1/2}$ scale.

\par The optimal sample size choice in~\eqref{eq - Optimal Sample size} balances
stochastic and approximation errors but does not satisfy this undersmoothing
condition. More generally, the requirement $\sqrt{N_\ell}\epsilon_\ell=o(1)$
depends on how the sieve dimension $n_\ell$ grows with the approximation error
$\epsilon_\ell$, which in turn is determined by the chosen partition of the
moment class $\mathcal{V}$. Loosely speaking, selecting $N_\ell$ slightly larger
than the bias--variance balance eliminates the approximation bias asymptotically.
While a full analysis of optimal tuning is beyond the scope of this paper, the
explicit form of the stochastic bounds provides clear guidance on how sample size
and sieve complexity interact.

\par Finally, the quantiles $q^{(l)}_{1-\delta/2,\ell}$ and
$q^{(u)}_{1-\delta/2,\ell}$ are unknown in practice and must be estimated from
the data. Any procedure yielding consistent estimators of these quantiles as
$\ell\to\infty$—such as Gaussian approximation or multiplier bootstrap methods
based on the empirical covariance structure of $\mathcal{G}_\ell$ in the spirit
of \cite{Chernozhukov-Anti-Concentration2014}—leads to asymptotically valid
confidence intervals for $\vartheta^*$.

\subsection{An Extension: Entropic Optimal Transport with Martingale Constraints}
Recent work, such as \cite{tang2025efficientalgorithmentropicoptimal}, has begun exploring entropically regularized martingale optimal transport problems from a computational perspective. These studies identify the existence, structure, and dual representations for entropic martingale transport, but leave open the question of statistical estimation under sampling.

The sieve M-estimation framework in this paper naturally extends to entropic martingale optimal transport. By formulating martingale and supermartingale constraints as conditional moment equalities or inequalities, and then converting them into (generally infinite) collections of unconditional moment restrictions using instrument functions, as in \cite{Tabri-MOOR-2025}, the resulting class of moment functions stays uniformly bounded and VC under Assumption 1, allowing the sieve approximation and stochastic optimization strategies presented here to handle these constraints and provide a principled estimation route.

However, to establish full statistical guarantees for constrained transport problems, we still need to develop additional tools. In particular, researchers have yet to find sharp bounds on the dual optimizers under general conditional moment restrictions--bounds that are crucial for controlling the stochastic program domains and for deriving uniform laws of large numbers or Gaussian approximations. Advancing these aspects remains an important area for future research.

More broadly, this extension demonstrates the flexibility of viewing entropically regularized optimal transport as a stochastic optimization problem under moment restrictions. By focusing on population-level constraint sets rather than sample-induced discretizations, my framework creates new opportunities for statistical inference in constrained transport problems--areas previously explored mainly through analytical or computational methods.

\subsection{Further Applications and Open Directions}

The above discussion shows that the proposed sieve M-estimation framework relies on the moment-inequality structure of entropically regularized optimal transport, rather than on the analytic properties of dual potentials. This focus opens the door to a wide range of applications where entropic optimal transport yields a well-defined $I$-projection, even if we lack detailed regularity or representation results for the corresponding Schr\"odinger potentials.

Recent literature illustrates the variety of constraints we can express as moment restrictions in entropic optimal transport. Beyond classical marginal constraints, martingale and supermartingale constraints play a central role in model-independent finance and stochastic control (see, e.g., \citealp{BeiglbockHenryLaborderePenkner2013,Beiglbock-Juillet,Nutz-Stebegg}, for supermartingale optimal transport; and  \citealp{tang2025efficientalgorithmentropicoptimal,NutzWieselMartingaleSB}, for entropic and Schr\"odinger bridge extensions). Convex order and dominance constraints can also be formulated as infinite families of linear moment inequalities indexed by convex test functions, which arise in risk aggregation and no-arbitrage pricing (\citealp{Strassen1965,Ruschendorf1985,BeiglbockPenkner2013}). Related moment-based approaches appear also in causal and adapted transport (\citealp{Lassalle2018,BackhoffVeraguasPammer2020}) and in Schr\"odinger bridge problems with structural or path-dependent constraints (\citealp{LEONARD}). In many of these cases, one can establish entropic minimizers and dual representations using the results of~\cite{Tabri-MOOR-2025}, but not sharp regularity results for the dual variables--such as continuity, boundedness, or smoothness.

My approach in this paper departs from traditional potential-based computational and analytical methods. Instead of parametrizing dual variables as functions and imposing continuity or smoothness for pointwise optimization, I operate directly on the population-level constraint set through moment inequalities. I treat dual variables as Radon measures acting on classes of moment functions and estimate them by approximating expectations in the Fenchel dual formulation. This method remains well defined even when Schr\"odinger potentials lack simple functional representations.

From this perspective, the lack of sharp results on dual potentials does not prevent statistical analysis of the corresponding entropic transport values. Instead, we need sufficient structural properties--such as boundedness of dual optimizers and VC complexity of the moment class--to control the stochastic approximation. Establishing these features for specific constrained entropic transport problems remains an important direction for future research.

Overall, this discussion highlights a conceptual shift: we can develop statistical optimal transport at the level of $I$-projections and moment restrictions, without requiring strong regularity assumptions on dual potentials. This approach enables statistical inference across a wide range of constrained entropic transport problems, many of which researchers currently study primarily from variational or computational perspectives.

\bibliographystyle{chicago}
\bibliography{mcgilletd}

\appendix
\section{Proofs of Results}
\subsection{Proposition~\ref{Thm - Fenchel}}\label{Proof- Fenchel}
\begin{proof}
\par {\bf Part 1}. We establish $\arginf\left\{\int_{\Omega}e^{z}\,dR_\gamma: z\in \mathcal{D}\right\}\neq\emptyset$ through the existence of an element being a cluster point of the sequence $\{z_{0,n}\}_{n\geq1}$, in the norm topology of $L_1(R_\gamma)$, where $$z_{0,n}=\arginf\left\{ \int_{\Omega}e^{y}\,dR_\gamma : z\in\mathcal{D}_n\right\}\quad\text{with}\quad\mathcal{D}_n=\left\{z\in\mathcal{D}:\alpha\leq \bar{\alpha}_n\right\}$$ for each $n$, and $\bar{\alpha}_n\nearrow\infty$ as $n\rightarrow\infty$. Lemma~\ref{Lemma - H-L} shows that I meet all of the conditions to apply Theorem~2.3 of~\citet{Mena-Lerma-2005}, so that 
\begin{align}
\text{OL}\left\{\{z_{0,n}\}_{n\geq1}\right\}\subset \arginf\left\{\int_{\Omega}e^{z}\,dR_\gamma: z\in \mathcal{D}\right\},
\end{align}
holds, where $\text{OL}\left\{\{z_{0,n}\}_{n\geq1}\right\}$ denotes the outer limit of $\{z_{0,n}\}_{n\geq1}$, which is also the set of cluster points of the sequence in the norm topology of $L_1(R_\gamma)$. Lemma~\ref{Lemma - Fenchel Dual existence of cluster point} establishes $\text{OL}\left\{\{z_{0,n}\}_{n\geq1}\right\}\neq\emptyset$. Hence, $\arginf\left\{\int_{\Omega}e^{z}\,dR_\gamma: z\in \mathcal{D}\right\}\neq\emptyset$. 

\par Now I shall establish the uniqueness a.s.$-R_\gamma$ of the minimizer. As any $v\in\mathcal{V}$ statisfies $\int_{\Omega}e^{v}\,dR_\gamma<\infty$, the minimizers cannot be where the objective function equals $\infty.$ Combining this implication with the strict convexity of the map $z\mapsto \int_{\Omega}e^{z}\,dR_\gamma$ on $\mathcal{D}$, implies that there is a unique minimizer (up to equivalence class). Let $\beta\in(0,1)$ and $z_1,z_2\in \mathcal{D}$ such that $z_1\neq z_2$ holds as equivalence classes. Additionally, let $z_3=\beta\, z_1+(1-\beta)\,z_2$. Then $\int_{\Omega}e^{z_3}\,dR_\gamma< \beta\int_{\Omega}e^{z_1}\,dR_\gamma+(1-\beta)\int_{\Omega}e^{z_3}\,dR_\gamma$, holds, by the strict convexity of the exponential function. This establishes the strict convexity of the map, and hence, the set of minimizers $\arginf\left\{\int_{\Omega}e^{z}\,dR_\gamma: z\in \mathcal{D}\right\}$, is unique up to equivalence class.

\par Next, I develop the representation of the $I$-projection's $R_\gamma$-density. From the above arguments let $z_{0,\gamma}=\arginf\left\{\int_{\Omega}e^{z}\,dQ: z\in \mathcal{D}\right\}$. The set $\mathcal{D}$ is convex, and the objective function $g(z)=\int_{\Omega}e^{z}\,dR_\gamma$ is G\^{a}teaux differentiable, then by Theorem~2 on page 178 of~\citet{Luenberger}, $\frac{d}{dt}g\left(z_{0,\gamma}+t(z-z_{0,\gamma})\right)\mid_{t=0}\,\geq0$ $\forall z\in\mathcal{D}$,
yielding $\int_{\Omega}(y-y_0)e^{y_0}\,dQ\geq0$ $\forall y\in\mathcal{D}$. By choosing $y=cy_0$ first with $c>1$ and then with $c<1$ (since $\mathcal{D}$ is also a cone), I obtain
\begin{align}\label{eq - proof thm 2 0}
\int_{\Omega}z_{0,\gamma}e^{z_{0,\gamma}}\,dR_\gamma=0,\;\text{and}\quad \int_{\Omega}ze^{z_{0,\gamma}}\,dR_\gamma\geq0\;\forall z\in\mathcal{D}.
\end{align}
Let $p_\gamma=\frac{e^{z_{0,\gamma}}}{\int_{\Omega}e^{z_{0,\gamma}}\,dR_\gamma}$, and note that the second part of~(\ref{eq - proof thm 2 0}) implies that $\int_{\Omega}v\,p_\gamma\,dR_\gamma\geq0\quad\forall v\in\mathcal{V}$; hence, $p_\gamma\in\mathcal{M}$. Furthermore, $m(p_\gamma)+\log\left(\int_{\Omega}e^{z_{0,\gamma}}\,dR_\gamma\right)=\int_{\Omega}z_{0,\gamma}e^{z_{0,\gamma}}\,dR_\gamma=0$, holds, by the first part of~(\ref{eq - proof thm 2 0}). Hence, by Theorem~2.2 of~\citet{Bhatacharya-Dykstra}, $p_\gamma$ solves the $I$-projection problem.

\par {\bf Part 2}. The proof proceeds by the direct method. Firstly, observe that $\exists v\in\mathcal{V}$ such that $\int_{\Omega}v\,dR_\gamma>0$; otherwise, the solution of the $I$-projection problem~(\ref{eq - KL Min Problem}) would be $R_\gamma$. Hence, by Part (ii) of Theorem~1 in \cite{Tabri-MOOR-2025},  $z_{0,\gamma}\in\overline{\text{span}_{+}(B)}$, where $\overline{\text{span}_{+}(B)}$ is the $L_1(R_\gamma)$-norm closure of the positive linear span of  \\ $B=\left\{v\in\mathcal{V}:\int_{\Omega}v\,dP_\gamma=0\right\}$. Now $B=\mathcal{V}$ must hold, as the $I$-projection problem is a moment inequality formulation of the equality constraints that set the marginal distributions. Therefore, I must establish that $y_0\in\overline{\text{span}_+(\mathcal{V})}$, holds. Since $R_\gamma$ violates the moments inequality restrictions, $z_{0,\gamma}\in\mathcal{D}$ implies that $z_{0,\gamma}=\alpha_{0,\gamma}\,z^\prime_{0,\gamma}$ with $\alpha_{0,\gamma}>0$ and $z^\prime_{0,\gamma}\in \overline{\text{co}}(\mathcal{V})$. The result of Lemma~\ref{Lemma - KMT and MT} implies that there are only two cases to consider in establishing the desired result: (i) $z^\prime_{0,\gamma}\in\text{ex}\left(\overline{\text{co}}(\mathcal{V})\right)$, and (ii) $z^\prime_{0,\gamma}\not\in\text{ex}\left(\overline{\text{co}}(\mathcal{V})\right)$, where $\text{ex}\left(\overline{\text{co}}(\mathcal{V})\right)$ denotes the set of extreme points of $\overline{\text{co}}(\mathcal{V})$. 

\par Starting with case (i), since $\text{ex}\left(\overline{\text{co}}(\mathcal{V})\right)\subset\mathcal{V}$ also by Lemma~\ref{Lemma - KMT and MT}, it must be that $z^\prime_{0,\gamma}\in\mathcal{V}$, and therefore, $z_{0,\gamma}\in\overline{\text{span}_+(\mathcal{V})}$. Next, consider case (ii): $z^\prime_{0,\gamma}\not\in\text{ex}\left(\overline{\text{co}}(\mathcal{V})\right)$. Then, $\exists n\in\mathbb{Z}_+$, $p_i>0$ for each $i=1,\ldots,n$ such that $\sum_{i=1}^{n}p_i=1$, for which $z_{0,\gamma}=\alpha_{0,\gamma} \sum_{i=1}^{n}p_iv_i$, where $\alpha_{0,\gamma}>0$ and $\left\{v_1,\ldots,v_n\right\}\subset\text{ex}\left(\overline{\text{co}}(\mathcal{V})\right)\subset\mathcal{V}$. Consequently, $z_{0,\gamma}\in\text{span}_+(\mathcal{V})\subset\overline{\text{span}_+(\mathcal{V})}$. 
\end{proof}

\subsection{Corollary~\ref{Coro - dual GF}}\label{Proof- Fenchel G-F}
\begin{proof}
The proof proceeds by the direct method. I only need to establish the equality \\ $\mathcal{S}^*=\left\{\xi\in\Upsilon: z_{0,\gamma}=\int_{\overline{\mathcal{V}}}v\,d\xi(v) \right\},$ as the non-emptiness trivially holds because $z_{0,\gamma}$ exists under the aforementioned conditions, and that $z_{0,\gamma}$ has a G-F integral representation~(\ref{eq - Result of Rudin Extended}). I start with the direction ``$\subset$''. Let $\xi^\prime\in\mathcal{S}^*$, then we must show that $z_{0,\gamma}=\int_{\mathcal{V}}v\,d\xi^\prime(v),$ holds. By optimality of $\xi^\prime$, $\int_{\Omega}e^{\int_{\mathcal{V}}v\,d\xi^\prime(v)}\,dR_\gamma\leq \int_{\Omega}e^{\int_{\mathcal{V}}v\,d\xi(v)}\,dR_\gamma\quad\forall \xi\in\Upsilon$, and that $\exists z^\prime\in\mathcal{D}$ such that $z^\prime=\int_{\mathcal{V}}v\,d\xi^\prime(v)$, observe that 
\begin{align}\label{eq - proof coro Fenchel GF 0}
\int_{\Omega}e^{z^\prime}\,dR_\gamma=\int_{\Omega}e^{\int_{\mathcal{V}}v\,d\xi^\prime(v)}\,dR_\gamma\leq \int_{\Omega}e^{\int_{\mathcal{V}}v\,d\xi(v)}\,dR_\gamma\quad\forall \xi\in\Upsilon.
\end{align}
Now by the G-F integral representation~(\ref{eq - Result of Rudin Extended}), I can rewrite~(\ref{eq - proof coro Fenchel GF 0}) as
\begin{align}
\int_{\Omega}e^{z^\prime}\,dR_\gamma=\int_{\Omega}e^{\int_{\mathcal{V}}v\,d\xi^\prime(v)}\,dR_\gamma\leq \int_{\Omega}e^z\,dR_\gamma\quad\forall z\in\mathcal{D},
\end{align} 
and hence, $\int_{\Omega}e^{\int_{\mathcal{V}}v\,d\xi^\prime(v)}\,dR_\gamma=\int_{\Omega}e^{z^\prime}\,dR_\gamma=\int_{\Omega}e^{z_{0,\gamma}}\,dR_\gamma$ must hold by Proposition~\ref{Thm - Fenchel}. Since $z_{0,\gamma}$ is unique (up to equivalence class) I must have $z_{0,\gamma}=z^{\prime}$ a.s.$-R_\gamma$, and hence,  $z_{0,\gamma}=\int_{\mathcal{V}}v\,d\xi^\prime(v)$ a.s.$-R_\gamma$ by the G-F integral representation~(\ref{eq - Result of Rudin Extended}). Therefore, $\xi^\prime\in\left\{\xi\in\Upsilon: z_{0,\gamma}=\int_{\mathcal{V}}v\,d\xi(v) \right\}$.

\par Next, I consider the direction ``$\supset$''. let $\xi^\prime\in\left\{\xi\in\Upsilon: z_{0,\gamma}=\int_{\mathcal{V}}v\,d\xi(v) \right\}$, then we must show that $\xi^\prime\in\mathcal{S}^*$. By the G-F integral representation~(\ref{eq - Result of Rudin Extended}) and the optimality of $z_{0,\gamma}$, observe that
\begin{align}\label{eq - proof coro Fenchel GF 1}
\int_{\Omega}e^{\int_{\mathcal{V}}v\,d\xi^\prime(v)}\,dR_\gamma=\int_{\Omega}e^{z_{0,\gamma}}\,dR_\gamma\leq \int_{\Omega}e^z\,dR_\gamma\quad\forall z\in\mathcal{D}.
\end{align}
 Now by the G-F integral representation~(\ref{eq - Result of Rudin Extended}), I can rewrite~(\ref{eq - proof coro Fenchel GF 1}) as $$\int_{\Omega}e^{\int_{\mathcal{V}}v\,d\xi^\prime(v)}\,dR_\gamma\leq \int_{\Omega}e^{\int_{\mathcal{V}}v\,d\xi(v)}\,dR_\gamma\quad\forall \xi\in\Upsilon,$$ and hence, $\xi^\prime\in\mathcal{S}^*$.
\end{proof}

\subsection{Proposition~\ref{prop - alpha UB}}\label{Proof- Prop prop - alpha UB}
\begin{proof}
The proof proceeds by the direct method. Under Assumption~\ref{Assump -Primities OTP}, the cost function, $c$, is bounded; i.e., $\|c\|_{\infty}<\infty$. The $I$-projection, $P_\gamma$, has density with respect to $P_X\otimes P_Y$ given by
\begin{align}
\frac{dP_\gamma}{d\left(P_X\otimes P_Y\right)}(x,y)=\exp\left\{\gamma\left(z_{0,\gamma}(x,y)/\gamma-c(x,y)\right)\right\}/(\vartheta^*\,a_\gamma).
\end{align}
Under Assumption~\ref{Assump -Primities OTP}, I can apply Theorem~2.1 in~\cite{nutz2022eot} to establish that this density is also that of the unique coupling $P\in\Pi(P_X,P_Y)$ having the form 
\begin{align}
\frac{dP_\gamma}{d\left(P_X\otimes P_Y\right)}(x,y)=\exp\left\{\gamma(\varphi_\gamma(x)+\psi_\gamma(y)-c(x,y))\right\}/(\vartheta^*\,a_\gamma),
\end{align}
with $\int_{\mathcal{X}}\varphi_\gamma(x)\,dP_X=\int_{\mathcal{Y}}\psi_\gamma(y)\,dP_Y$, where $\varphi_\gamma\in L_1(P_X)$ and $\psi_\gamma\in L_1(P_Y)$ are the unique solution of the EOT's dual problem, given by
\begin{align*}
\sup_{f\in L_1(P_X),\, g\in L_1(P_Y)}
& \left(
\int_{\mathcal{X}} f\,dP_X
+
\int_{\mathcal{Y}} g\,dP_Y
\right.\\
&\qquad\left.-
\frac{1}{\gamma}
\Big(
\int_{\mathcal{X}\times\mathcal{Y}}
e^{\gamma(f(x)+g(y)-c(x,y))}
\, d(P_X\otimes P_Y)
- 1
\Big)
\right).
\end{align*}
Because $a_\gamma\leq 1$, holds, the above normalization implies $\int_{\mathcal{X}}\varphi_\gamma(x)\,dP_X,\int_{\mathcal{Y}}\psi_\gamma(y)\,dP_Y\geq 0$, under Assumption~\ref{Assump -Primities OTP}; see Remark~4.10 in~\cite{nutz2022eot}. 

\par In consequence, I obtain $z_{0,\gamma}(x,y)/\gamma=\varphi_\gamma(x)+\psi_\gamma(y)$ upon equating the two forms of the density $\frac{dP_\gamma}{d\left(P_X\otimes P_Y\right)}$. Now because $z_{0,\gamma}=\alpha_{0,\gamma}\,\int_{\mathcal{V}}v\,d\mu_{0,\gamma}(v)$, integrating both sides with respect to $P_X\otimes P_Y$ yields, $\int_{\mathcal{X}\times\mathcal{Y}}z_{0,\gamma}\,d\left(P_X\otimes P_Y\right)=0$, implying the desired normalization
\begin{align}
0=\int_{\mathcal{X}\times\mathcal{Y}}\left(\varphi(x)+\psi(y)\right)\,d\left(P_X\otimes P_Y\right)=\int_{\mathcal{X}}\varphi(x)\,dP_X+\int_{\mathcal{Y}}\psi(y)\,dP_Y.
\end{align}
Note that $\int_{\mathcal{X}\times\mathcal{Y}}z_{0,\gamma}\,d\left(P_X\otimes P_Y\right)=0$ follows from arguments in the proof of Part 2 of Proposition~\ref{Thm - Fenchel} that establish $z_{0,\gamma}=\alpha_{0,\gamma}z^\prime_{0,\gamma}$ with $\alpha_{0,\gamma}>0$ and $z^\prime_{0,\gamma}\in\text{ex}\left(\overline{\text{co}}(\mathcal{V})\right)$ or $z^\prime_{0,\gamma}\not\in\text{ex}\left(\overline{\text{co}}(\mathcal{V})\right)$. Observe that 
$\int_{\mathcal{X}\times\mathcal{Y}}z^\prime_{0,\gamma}\,d\left(P_X\otimes P_Y\right)=0$, holds, in either case as the moment functions had the form described in~(\ref{eq - moment fxns}).

\par Re-writing $z_{0,\gamma}/\gamma=\varphi_\gamma+\psi_\gamma$ as 
$(\alpha_{0,\gamma})/\gamma\,\int_{\overline{\mathcal{V}}}v\,d\mu_{0,\gamma}(v)=\varphi_\gamma+\psi_\gamma$, it follows that $\int_{\overline{\mathcal{V}}}v\,d\mu_{0,\gamma}(v)=\varphi_\gamma'+\psi_\gamma'$, where $\varphi_\gamma'=(\gamma/\alpha_{0,\gamma})\varphi_\gamma$ and
$\psi_\gamma'=(\gamma/\alpha_{0,\gamma})\psi_\gamma$. Now because $\int_{\overline{\mathcal{V}}}v\,d\mu_{0,\gamma}(v)\in\overline{\text{co}}(\mathcal{V})$, I must have that $\varphi_\gamma',\psi_\gamma',\in [-1,1]$. Hence, by Lemma~4.9 in~\cite{nutz2022eot}, $(\alpha_{0,\gamma}/\gamma)\varphi_\gamma'(x)\leq \|c\|_{\infty}$ for each $x\in\mathcal{X}$. Then, for $x$ such that $\varphi_\gamma'(x)>0$, $(\alpha_{0,\gamma}/\gamma)\leq \|c\|_{\infty}/\varphi_\gamma'(x)$, and therefore, $(\alpha_{0,\gamma}/\gamma)\leq \inf\left\{\|c\|_{\infty}/\varphi_\gamma'(x):\varphi_\gamma'(x)>0\right\}$. Now observe that
\begin{align*}
\inf\left\{\|c\|_{\infty}/\varphi_\gamma'(x):\varphi_\gamma'(x)>0\right\}=\frac{\|c\|_{\infty}}{\sup_{\{x:\varphi_\gamma'(x)>0\}}\varphi_\gamma'(x)},
\end{align*}
and because $\varphi_\gamma'$ must be a linear combination of functions of the form $x\mapsto F_{X}(x^{\prime})-1\left[x\preceq x^{\prime}\right]$ or $x\mapsto1\left[x\preceq x^{\prime}\right]-F_{X}(x^{\prime})$, where the coefficients are positive and bounded from above by 1, 
$$\frac{\|c\|_{\infty}}{\sup_{\{x:\varphi_\gamma'(x)>0\}}\varphi_\gamma'(x)}\leq\frac{\|c\|_{\infty}}{1-F_X(\inf\mathcal{X})}.$$
The same line of arguments holds if I also use $\psi_\gamma'$ instead of $\varphi_\gamma'$, so that 
$$\frac{\|c\|_{\infty}}{\sup_{\{x:\varphi_\gamma'(x)>0\}}\varphi_\gamma'(x)}\leq\frac{\|c\|_{\infty}}{\max\left\{1-F_X(\inf\mathcal{X}),1-F_Y(\inf\mathcal{Y})\right\}}$$
must hold. Hence $(\alpha_{0,\gamma}/\gamma)\leq\frac{\|c\|_{\infty}}{\max\left\{1-F_X(\inf\mathcal{X}),1-F_Y(\inf\mathcal{Y})\right\}}$, as desired.
\end{proof}

\subsection{Proposition~\ref{Thm - Computation}}\label{Proof Prop computation}
\begin{proof}
\noindent The proof proceeds by the direct method. 
\par {\bf Part 1}. Let $\xi_{0,\ell}=\alpha_{0,\gamma}\sum_{i=1}^{n_\ell}\mu_{i,0,\ell}\,\delta_{v_i}$ such that $\mu_{i,0,\ell}=\mu_{0,\gamma}\left(E_{i,\ell}\right)$ for each $i=1,\ldots,n_\ell$. By optimality of $\xi_{0,\gamma}$ and $(\alpha^*_\ell,\mu^*_\ell)\in\mathcal{S}_\ell^*$, I have the following inequalities
\begin{align}\label{Proof - Comp -1}
\vartheta^*=\int_{\Omega}e^{\int_{\overline{\mathcal{V}}}v\,d\xi_{0,\gamma}(v)}\,dR_\gamma\leq \vartheta_\ell^*= \int_{\Omega}G_\ell(\omega;\alpha^*_\ell,\mu^*_\ell)\,dR_\gamma  \leq \int_{\Omega}e^{\int_{\overline{\mathcal{V}}}v\,d\xi_{0,\ell}(v)}\,dR_\gamma.
\end{align}
Then, 
\begin{align}
\left|\vartheta^*-\int_{\Omega}e^{\int_{\overline{\mathcal{V}}}v\,d\xi_{0,\ell}(v)}\,dR_\gamma\right| & =\left|\int_{\Omega}e^{\int_{\overline{\mathcal{V}}}v\,d\xi_{0,\gamma}(v)}\,dR_\gamma-\int_{\Omega}e^{\int_{\mathcal{V}}v\,d\xi_{0,\ell}(v)}\,dR_\gamma\right|\label{eq - Optimality Lipschitz 1} \\
& \leq e^{\alpha_{0,\gamma}}\alpha_{0,\gamma}\left\|\int_{\mathcal{V}}v\,d\mu_{0,\gamma}(v)-\int_{\mathcal{V}}v\,d\mu_{0,\ell}(v)\right\|_{L_1(R_\gamma)}\nonumber \\
&\leq\epsilon_\ell \,\gamma\kappa\|c\|_\infty\,e^{\gamma\kappa\|c\|_\infty},\label{eq - Optimality Lipschitz 2}
\end{align}
 by applying the Mean Value Theorem to the discrepancy~(\ref{eq - Optimality Lipschitz 1}) with respect to the exponential function, and then using the result of Proposition~\ref{prop - alpha UB} and the fact that 
$\int_{\mathcal{V}}v\,d\mu_{0,\gamma}(v),\int_{\mathcal{V}}v\,d\mu_{0,\ell}(v)\in[-1,1]$ a.s.$-R_\gamma$, holds, to deduce the upper bound~(\ref{eq - Optimality Lipschitz 2}). This derivation holds for each $\ell$, and hence, \\ $\lim_{\ell\rightarrow\infty}\int_{\Omega}e^{\int_{\mathcal{V}}v\,d\xi_{0,\ell}(v)}\,dR_\gamma=\vartheta^*$, because $\lim_{\ell\rightarrow\infty}\epsilon_\ell=0$. Finally, from the inequalities~(\ref{Proof - Comp -1}), the above limit implies $\lim_{\ell\rightarrow\infty}\vartheta_\ell^*=\vartheta^*$.

\par {\bf Part 2}. By the inequalities~(\ref{Proof - Comp -1}), observe that for each $\ell$,  
\begin{align*}
\left|\log\vartheta^*_\ell-\log\vartheta^*\right|=\log\vartheta^*_\ell-\log\vartheta^*= \log\left(1+\frac{\vartheta^*_\ell-\vartheta^*}{\vartheta^*}\right) & \leq \frac{\vartheta^*_\ell-\vartheta^*}{\vartheta^*}\\
 & \leq \epsilon_\ell\kappa\gamma\|c\|_{\infty}e^{\gamma(\kappa\|c\|_{\infty}+\bar{c})+\log a_\gamma},
\end{align*}
where the second inequality arises from bounding $\vartheta^*$ from below by way of Lemma~\ref{Lemma - LB on Dual value}.
 
\par {\bf Part 3}. The proof follows steps identical to those in Part 2 of Theorem 4 in~\cite{Tabri-MOOR-2025}. I present them here for completeness. Let $\xi^\prime$ be an accumulation point of $\{\xi^*_\ell\}_{\ell\geq1}$ in the weak-star topology of $C\left(\mathcal{V}\right)^*$. Therefore, there exists a subsequence $\{\xi^*_{\ell_h}\}_{h\geq1}$ such that $\xi^*_{\ell_h}\stackrel{w*}{\longrightarrow}\xi^\prime$. By observing that
\begin{align*}
\vartheta^*=\int_{\Omega}e^{\int_{\mathcal{V}}v\,d\xi_{0,\gamma}(v)}\,dR_\gamma\leq \int_{\Omega}e^{\int_{\mathcal{V}}v\,d\xi_{\ell_h}(v)}\,dR_\gamma & =  \int_{\Omega}e^{\alpha_{\ell_h}\sum_{i=1}^{n_{\ell_h}}\mu_{i,\ell_h}\,v_i}\,dR_\gamma \\
& \leq\int_{\Omega}e^{\alpha_{0,\gamma}\sum_{i=1}^{n_{\ell_h}}\mu_{0}(E_{i,{\ell_h}})\,v_i}\,dR_\gamma,
\end{align*}
holds, Part 1 of this proposition implies $\lim_{h\rightarrow\infty}\int_{\Omega}e^{\int_{\mathcal{V}}v\,d\xi_{\ell_{h}}}\,dR_\gamma=\int_{\Omega}e^{\int_{\mathcal{V}}v\,d\xi_{0,\gamma}(v)}\,dR_\gamma$. Now using the fact that the map $\xi\mapsto\int_{\overline{\mathcal{V}}}v\,d\xi$ is continuous when $C\left(\mathcal{V}\right)^*$ is given the weak-star topology and $L_1(R_\gamma)$ has the weak topology, it follows that $\int_{\mathcal{V}}v\,d\xi_{\ell_{h}}\stackrel{w}{\longrightarrow}\int_{\mathcal{V}}v\,d\xi^{\prime}$. Applying the Skorohod Representation Theorem, there exists a probability space $(\Omega^\prime,\mathcal{F}^\prime,Q^\prime)$ and real-valued measurable functions $\{z_h\}_{h\geq 1}$ and $z$ on this probability space, such that
\begin{enumerate}[(a)]
\item $\{z_h\}_{h\geq1}$ and $z$ have the same probability distributions as $\{e^{\int_{\mathcal{V}}v\,d\xi_{\ell_h}}\}_{h\geq 1}$ and $e^{\int_{\mathcal{V}}v\,d\xi^\prime}$, respectively, and
\item $\lim_{h\rightarrow\infty}z_h(\omega^\prime)=z(\omega^\prime)$ a.s.$-Q^\prime$.
\end{enumerate}
One can set $\Omega^\prime=[0,1]$ with $\{z_h\}_{h\geq1}$ and $z$ being the quantile functions of $\{e^{\int_{\mathcal{V}}v\,d\xi_{\ell_h}}\}_{h\geq 1}$ and $e^{\int_{\mathcal{V}}v\,d\xi^\prime}$, respectively. Consequently, $z_h(\omega^{\prime}),z(\omega^{\prime})\geq 0$ for all $\omega^{\prime}\in [0,1]$. Now we can use Fatou's Lemma to deduce the desired result:
\begin{align*}
\int_{\Omega}e^{\int_{\mathcal{V}}v\,d\xi^\prime}\,dR_\gamma=\int_{[0,1]}z\,dQ^{\prime}\leq\liminf_{h\rightarrow\infty}\int_{[0,1]}z_h\,dQ^{\prime} & =\liminf_{h\rightarrow\infty}\int_{\Omega}e^{\int_{\mathcal{V}}v\,d\xi_{\ell_h}}\,dR_\gamma \\
& =\int_{\Omega}e^{\int_{\mathcal{V}}v\,d\xi_{0,\gamma}(v)}\,dR_\gamma.
\end{align*} 
Therefore, $\xi^{\prime}\in\arginf\left\{\int_{\Omega}e^{\int_{\mathcal{V}}v\,d\xi(v)}\,dR_\gamma:\xi\in\Upsilon\right\}$ since $\xi_{0,\gamma}$ is an element of that set. This concludes the proof.
\end{proof}

\subsection{Proof of Theorem~\ref{Thm Stochastic Opt Sieve}}\label{Proof - Thm Stoch Opt Sieve}
\begin{proof}
The proof proceeds by the direct method. I firstly introduce some notation and set the stage to simplify the exposition. Recall that $\xi_{0,\gamma}=\alpha_{0,\gamma}\cdot\mu_{0,\gamma}\in\mathcal{S}^*$. Given $\ell$, and hence, $\epsilon_\ell$, there corresponds a finite partition $\{E_{i,\ell}\}_{i=1}^{n_\ell}$ of $\mathcal{V}$ with the accuracy~(\ref{eq - Disc Accuracy}), and let $v_i\in E_{i,\ell}$ for $i=1,\ldots, n_\ell$. Define $\xi_{0,\ell}=\alpha_{0,\gamma}\sum_{i=1}^{n_\ell}\mu_{i,0,\ell}\delta_{v_i}$, where $\mu_{i,0,\ell}=\mu_{0,\gamma}\left(E_{i,\ell}\right)$ and $\delta_{v_i}$ is the Dirac delta function at $v_i$, for each $i=1,\ldots, n_\ell$. Now based on this partition of $\mathcal{V}$, let $(\alpha^*_\ell,\mu^*_\ell)\in\mathcal{S}_\ell^*$, $(\alpha_\ell,\mu_\ell)\in S_\ell$ and $(\hat{\alpha}_\ell,\hat{\mu}_\ell)\in\hat{S}_\ell$. 

\par Note that for each $\ell$, $\hat{S}_\ell\neq\emptyset$ for almost every sample realization of $\omega^N_\ell$ under $R_\gamma$. The realizations of the sample $\omega^{N_\ell}$ where $\hat{S}_\ell=\emptyset$ arises is on a set of probability measure zero. This is because the Extreme Value Theorem applies for each realization of $\omega^{N_\ell}$, as $\overline{\mathcal{C}}_\ell$ is a nonempty compact subset of $\mathbb{R}^{n_\ell+1}$ and the objective function is a Carath\'{e}odory function -- see Remark~\ref{Remark - Measurability SAA}.

\par {\bf  Part 1(i)}. Following the proof of Part (i) in Proposition~\ref{Thm - Computation}, using the optimality of $\xi_{0,\gamma}$ and $(\alpha^*_\ell,\mu^*_\ell)$, I have the inequalities
\begin{align}\label{Proof - Comp a.s -1}
\vartheta^*=\int_{\Omega}e^{\int_{\mathcal{V}}v\,d\xi_0(v)}\,dR_\gamma\leq \vartheta^*_\ell\leq \vartheta_\ell\leq \int_{\Omega}e^{\int_{\mathcal{V}}v\,d\xi_{0,\ell}(v)}\,dR_\gamma.
\end{align}
Because $\lim_{\ell\rightarrow\infty}\int_{\Omega}e^{\int_{\mathcal{V}}v\,d\xi_{0,\ell}(v)}\,dR_\gamma=\vartheta^*$ holds by Part 1 in Proposition~\ref{Thm - Computation}, by a sandwiching argument with the inequalities~(\ref{Proof - Comp a.s -1}), it follows that  $\lim_{\ell\rightarrow\infty}\vartheta_\ell=\lim_{\ell\rightarrow\infty}\vartheta^*_\ell=\vartheta^*$, must also hold. Now along realizations $\omega^{N_\ell}$ such that $\hat{S}_\ell\neq\emptyset$,
I can add and subtract $\hat{\vartheta}_\ell$ and $\vartheta_\ell$ in the center of the chain of inequalities~(\ref{Proof - Comp a.s -1})
to obtain
\begin{align*}
\int_{\Omega}e^{\int_{\mathcal{V}}v\,d\xi_0(v)}\,dR_\gamma & \leq\vartheta^*_\ell-\vartheta_\ell+\vartheta_\ell-\hat{\vartheta}_\ell+\hat{\vartheta}_\ell\leq \int_{\Omega}e^{\int_{\mathcal{V}}v\,d\xi_{0,\ell}(v)}\,dR_\gamma.
\end{align*}
Hence, to obtain the desired result, I must establish that $\vartheta_\ell-\hat{\vartheta}_\ell=o_{a.s}(1)$, holds.

\par Toward that end, note that $\vartheta_\ell-\hat{\vartheta}_\ell$ is given by
\begin{align}\label{Proof - Comp 0}
\int_{\Omega}G_\ell(\omega;\alpha_\ell,\mu_\ell)\,dR_\gamma-\frac{1}{N_\ell}\sum_{j=1}G_\ell(\omega_j;\hat{\alpha}_\ell,\hat{\mu}_\ell),
\end{align}
I shall bound this term from above and from below by $o_{a.s}(1)$ random variables. The upper bound is
\begin{align}\label{Proof - Comp -2}
\int_{\Omega}G_\ell(\omega;\hat{\alpha}_\ell,\hat{\mu}_\ell)\,dR_\gamma-\frac{1}{N_\ell}\sum_{j=1}G_\ell(\omega_j;\hat{\alpha}_\ell,\hat{\mu}_\ell),
\end{align}
as $\int_{\Omega}G_\ell(\omega;\hat{\alpha}_\ell,\hat{\mu}_\ell)\,dR_\gamma\geq \int_{\Omega}G_\ell(\omega;\alpha_\ell,\mu_\ell)\,dR_\gamma$. Now the absolute value of the term in~(\ref{Proof - Comp -2}) is bounded from above by
\begin{align}\label{Proof - Comp -3}
\sup_{(\alpha,\mu)\in \overline{\mathcal{C}}_{\ell}}\left|\int_{\Omega}G_\ell(\omega;\alpha,\mu)\,dR_\gamma-\frac{1}{N_\ell}\sum_{j=1}G_\ell(\omega_j;\alpha,\mu)\right|,
\end{align}
From Lemma~\ref{Lemma - Rademacher complexity G}, and Corollary~\ref{Corollary - Rademacher complexity}, the Rademacher complexity of $\mathcal{G}_\ell$, $\mathcal{R}_{N_\ell}(\mathcal G_\ell)$, satisfies
$\mathcal{R}_{N_\ell}(\mathcal G_\ell)\leq \gamma\kappa\|c\|_{\infty}\,e^{\gamma\kappa\|c\|_{\infty}}\sqrt{\frac{2\log n_\ell}{N_\ell}}\quad \forall \ell,$ and hence, it vanishes under the limit condition~(\ref{eq - entropy bound}). Therefore, by an application of by Proposition~\ref{Prop - ULLN}, the term~(\ref{Proof - Comp -3}) must be $o_{a.s}(1)$. For the lower bound, observe that the term in~(\ref{Proof - Comp 0}) is bounded from below by
\begin{align}\label{Proof - Comp -4}
\int_{\Omega}G_\ell(\omega;\alpha_\ell,\mu_\ell)\,dR_\gamma-\frac{1}{N_\ell}\sum_{j=1}G_\ell(\omega_j;\alpha_\ell,\mu_\ell)
\end{align}
as $\frac{1}{N_\ell}\sum_{j=1}G_\ell(\omega_j;\alpha_\ell,\mu_\ell)\geq \frac{1}{N_\ell}\sum_{j=1}G_\ell(\omega_j;\hat{\alpha}_\ell,\hat{\mu}_\ell)$. Now the absolute value of the term in~(\ref{Proof - Comp -4}) is bounded from above by the term in~(\ref{Proof - Comp -3}), so that it is also $o_{a.s}(1)$ by the same arguments. In consequence,
$\vartheta^*=\int_{\Omega}e^{\int_{\overline{\mathcal{V}}}v\,d\xi_0(v)}\,dR_\gamma \leq o(1)+o_{a.s}(1)+\hat{\vartheta}_\ell\leq \int_{\Omega}e^{\int_{\overline{\mathcal{V}}}v\,d\xi_{0,\ell}(v)}\,dR_\gamma$,
 and taking the limit as $\ell\rightarrow\infty$ on all sides of these inequalities yields $\lim_{\ell\rightarrow\infty}\hat{\vartheta}_\ell=\vartheta^*$ a.s.$-R_\gamma$.

\par {\bf Part 1(ii)}. Suppose that $\{\hat{\xi}_{\ell_h}\}_{h\geq 1}$ converges to $\xi^\prime$ in the weak-star topology of $C(\mathcal{V})^*$ as $h\rightarrow\infty$, with probability 1. As the sequence $\{\hat{\vartheta}_{\ell_h}\}_{h\geq 1}$ is a subsequence of $\{\hat{\vartheta}_\ell\}_{\ell\geq 1}$, the result of Part 1(i) of this theorem implies $\lim_{h\rightarrow\infty}\hat{\vartheta}_{\ell_h}=\vartheta^*$ a.s.$-R_\gamma$. In consequence, I must establish the limit
$\lim_{h\rightarrow\infty}\hat{\vartheta}_{\ell_h}=\int_{\Omega}e^{\int_{\overline{V}}v\,d\xi^\prime}\,dR_\gamma$ with probability 1, to obtain the desired result, as it implies $\int_{\Omega}e^{\int_{V}v\,d\xi^\prime}\,dR_\gamma=\vartheta^*$, holds, and hence, $\xi^\prime\in\mathcal{S}^*$.

\par By using the fact that the map $\xi\mapsto\int_{\overline{\mathcal{V}}}v\,d\xi$ is continuous when $C\left(\mathcal{V}\right)^*$ is given the weak-star topology and $L_1(R_\gamma)$ has the weak topology, it follows that $\left\{\int_{\mathcal{V}}v\,d\hat{\xi}_{\ell_{h}}\right\}_{h\geq1}$ converges to $\int_{\mathcal{V}}v\,d\xi^{\prime}$ in the weak topology of $L_1(R_\gamma)$, with probability 1. I can apply the Skorohod Representation Theorem as follows: for almost every realization, there exists a probability space $(\Omega^\prime,\mathcal{F}^\prime,Q^\prime)$ and real-valued measurable functions $\{z_h\}_{h\geq 1}$ and $z$ on this probability space, such that
\begin{enumerate}[(a)]
\item $\{z_h\}_{h\geq1}$ and $z$ have the same probability distributions as $\{e^{\int_{\mathcal{V}}v\,d\hat{\xi}_{\ell_h}}\}_{h\geq 1}$ and $e^{\int_{\mathcal{V}}v\,d\xi^\prime}$, respectively, and
\item $\lim_{h\rightarrow\infty}z_\ell(\omega^\prime)=z(\omega^\prime)$ a.s.$-Q^\prime$.
\end{enumerate}
One can set $\Omega^\prime=[0,1]$ with $\{z_h\}_{h\geq1}$ and $z$ being the quantile functions of $\{e^{\int_{\mathcal{V}}v\,d\hat{\xi}_{\ell_h}}\}_{h\geq 1}$ and $e^{\int_{\mathcal{V}}v\,d\xi^\prime}$, respectively. Consequently, $z_h(\omega^{\prime}),z(\omega^{\prime})\geq 0$ for all $\omega^{\prime}\in [0,1]$. Now applying Fatou's lemma on a per realization basis, for almost every realization,  
\begin{align*}
\int_{\Omega}e^{\int_{\mathcal{V}}v\,d\xi^\prime}\,dR_\gamma=\int_{[0,1]}z\,dQ^{\prime}\leq\liminf_{h\rightarrow\infty}\int_{[0,1]}z_h\,dQ^{\prime} & =\liminf_{h\rightarrow\infty}\int_{\Omega}e^{\int_{\overline{\mathcal{V}}}v\,d\hat{\xi}_{\ell_h}}\,dR_\gamma \\
& =\int_{\Omega}e^{\int_{\mathcal{V}}v\,d\xi_{0,\gamma}(v)}\,dR_\gamma.
\end{align*}
Therefore, $\xi^{\prime}\in\arginf\left\{\int_{\Omega}e^{\int_{\mathcal{V}}v\,d\xi(v)}\,dR_\gamma:\xi\in\Upsilon\right\}$ because $\xi_{0,\gamma}$ is an element of that set. This concludes the proof.

\par {\bf Part 2}. By the triangle inequality, $\left|\log\hat{\vartheta}_\ell -\log\vartheta^*\right|$ is bounded from above by
\begin{align*}
\left|\log\hat{\vartheta}_\ell-\log\vartheta_\ell\right|+\left|\log\vartheta_\ell-\log\vartheta^*_\ell\right|+\left|\log\vartheta^*_\ell-\log\vartheta^*\right|.
\end{align*}
The last two terms in the above display do not depend on the sample, and hence, are non-stochastic. I firstly focus on these non-stochastic terms. By Part 2 of Proposition~\ref{Thm - Computation}, $$\left|\log\vartheta^*_\ell-\log\vartheta^*\right|=\log\vartheta^*_\ell-\log\vartheta^*\leq \epsilon_\ell\gamma\kappa\|c\|_{\infty}e^{\gamma(\kappa\|c\|_{\infty}+\bar{c})+\log a_\gamma}.$$
Furthermore,$\left|\log\vartheta_\ell-\log\vartheta^*_\ell\right|= \log\vartheta^*_\ell-\log\vartheta_\ell=\log\left(1+\frac{\vartheta^*_\ell-\vartheta_\ell}{\vartheta_\ell}\right)\leq \frac{\vartheta^*_\ell-\vartheta_\ell}{\vartheta_\ell}\leq \frac{\vartheta^*_\ell-\vartheta^*}{\vartheta^*},$
because $\vartheta^*\leq \vartheta_\ell$ for each $\ell$, and hence, by Part 2 of Proposition~\ref{Thm - Computation} $\left|\log\vartheta_\ell-\log\vartheta^*_\ell\right|\leq \epsilon_\ell\gamma\kappa\|c\|_{\infty}e^{\gamma(\kappa\|c\|_{\infty}+\bar{c})+\log a_\gamma}$, as well.

\par I turn my focus now to bounding $E_{R_\gamma}^{\otimes_{N_\ell}}\left[\left|\log\hat{\vartheta}_\ell-\log\vartheta_\ell\right|\right]$. Firstly, observe that by a first-order Taylor expansion of the logarithm function, $\left|\log\hat{\vartheta}_\ell-\log\vartheta_\ell\right| = \frac{|\hat{\vartheta}_\ell-\vartheta_\ell|}{\lambda\hat{\vartheta}_\ell+(1-\lambda)\vartheta_\ell}$,
where $\lambda\in[0,1]$. As $(\alpha,\mu)\in\overline{\mathcal{C}_\ell}$ and $v\in[-1,1]$ for each sample and $\ell$, it follows that $\hat{\vartheta}_\ell\geq e^{-\gamma\kappa\|c\|_{\infty}}$ for each sample and $\ell$. Furthermore, $\vartheta_\ell \geq \vartheta^*$ for each $\ell$ by the inequality in~(\ref{Proof - Comp a.s -1}), and  $\vartheta^*\geq e^{-(\gamma \bar{c}+\log a_\gamma)}$ by Lemma~\ref{Lemma - LB on Dual value}. As $\kappa\|c\|_{\infty}\geq\bar{c}\geq  \bar{c}+\gamma^{-1}\log a_\gamma$, $\left|\log\hat{\vartheta}_\ell-\log\vartheta_\ell\right| \leq  e^{\gamma\kappa\|c\|_{\infty}}|\hat{\vartheta}_\ell-\vartheta_\ell|$
holds for each sample and $\ell$, implying $E_{R_\gamma}^{\otimes_{N_\ell}}\left[\left|\log\hat{\vartheta}_\ell-\log\vartheta_\ell\right|\right]\leq e^{\gamma\kappa\|c\|_{\infty}}E_{R_\gamma}^{\otimes_{N_\ell}}\left[|\hat{\vartheta}_\ell-\vartheta_\ell|\right]$. Now, 
\begin{align*}
e^{\gamma\kappa\|c\|_{\infty}}E_{R_\gamma}^{\otimes_{N_\ell}}\left[|\hat{\vartheta}_\ell-\vartheta_\ell|\right]& \leq e^{\gamma\kappa\|c\|_{\infty}}E_{R_\gamma^{\otimes_{N_\ell}}}\left[\sup_{(\alpha,\mu)\in \overline{\mathcal{C}}_\ell}\left|\frac{1}{N_\ell}\sum_{j=1}G_\ell(\omega_j,\alpha,\mu)-\int_{\Omega}G_\ell(\omega,\alpha,\mu)\right|\right]\\
& \leq e^{\gamma\kappa\|c\|_{\infty}}2\mathcal{R}_{N_\ell}(\mathcal{G}_\ell)\leq 2\gamma\kappa\|c\|_{\infty}\sqrt{\frac{2\log n_\ell}{N_\ell}}\,e^{2\gamma\kappa\|c\|_{\infty}},
\end{align*}
with the second and third inequalities, holding, by applications of Lemma~ 2.3.1 in~\cite{VDV-W} (i.e. the Symmetrization Lemma), and Lemma~\ref{Lemma - Rademacher complexity G} and Corollary~\ref{Corollary - Rademacher complexity}, respectively. 
Therefore, $\gamma^{-1}E_{R_\gamma}^{\otimes_{N_\ell}}\left[\left|\log\hat{\vartheta}_\ell-\log\vartheta_\ell\right|\right]\leq 2\kappa\|c\|_{\infty}\sqrt{\frac{2\log n_\ell}{N_\ell}}\,e^{2\gamma\kappa\|c\|_{\infty}}$.
Finally, combining these bounds from the triangular inequality, I obtain
\begin{align*}
\gamma^{-1}E_{R_\gamma}^{\otimes_{N_\ell}}\left[\left|\log\hat{\vartheta}_\ell -\log\vartheta^*\right|\right]& \leq 2\kappa\|c\|_{\infty}\sqrt{\frac{2\log n_\ell}{N_\ell}}\,e^{2\gamma\kappa\|c\|_{\infty}}+2\epsilon_\ell\gamma\kappa\|c\|_{\infty}e^{\gamma(\kappa\|c\|_{\infty}+\bar{c})+\log a_\gamma}\\
&\leq  2\kappa\|c\|_{\infty}\sqrt{\frac{2\log n_\ell}{N_\ell}}\,e^{2\gamma\kappa\|c\|_{\infty}}+2\epsilon_\ell\gamma\kappa\|c\|_{\infty}e^{\gamma2\kappa\|c\|_{\infty}}\\
&\leq \max\left\{\sqrt{\frac{2\log n_\ell}{N_\ell}},\epsilon_\ell\,\right\}2\kappa\|c\|_{\infty}e^{\gamma2\kappa\|c\|_{\infty}}.
\end{align*}

\par {\bf Part 3}. Starting with the lower bound, observe that 
\begin{align*}
\sqrt{N_\ell}(\hat{\vartheta}_\ell-\vartheta^*)=\sqrt{N_\ell}(\hat{\vartheta}_\ell-\vartheta_\ell+\vartheta_\ell-\vartheta^*)\leq \sqrt{N_\ell}(\hat{\vartheta}_\ell-\vartheta_\ell)+\sqrt{N_\ell}\epsilon_\ell\,\gamma\kappa\|c\|_{\infty}e^{\gamma\kappa\|c\|_{\infty}}
\end{align*}
holds by Part 2 of Proposition~\ref{Thm - Computation}. I can re-write this inequality as
\begin{align}
\sqrt{N_\ell}\vartheta^* & \geq \sqrt{N_\ell}\hat{\vartheta}_\ell + \sqrt{N_\ell}(\vartheta_\ell-\hat{\vartheta}_\ell)-\sqrt{N_\ell}\epsilon_\ell\,\gamma\kappa\|c\|_{\infty}e^{\gamma\kappa\|c\|_{\infty}}\nonumber \\
 &= \sqrt{N_\ell}\hat{\vartheta}_\ell-\sqrt{N_\ell}\epsilon_\ell\,\gamma\kappa\|c\|_{\infty}e^{\kappa\gamma\|c\|_{\infty}}\nonumber \\ & \qquad+\sqrt{N_\ell}\left(\inf_{(\alpha,\mu)\in\overline{C}_\ell}E_{R_\gamma}\left[G_\ell(\alpha,\mu,\omega)\right]-\inf_{(\alpha,\mu)\in\overline{C}_\ell}\frac{1}{N_\ell}\sum_{j=1}^{N_\ell}G_\ell(\alpha,\mu,\omega_j)\right)\nonumber\\
 & \geq \sqrt{N_\ell}\hat{\vartheta}_\ell-\sqrt{N_\ell}\epsilon_\ell\,\gamma\kappa\|c\|_{\infty}e^{\kappa\gamma\|c\|_{\infty}}\nonumber\\
 &\qquad+\sqrt{N_\ell}\inf_{(\alpha,\mu)\in\overline{C}_\ell}\left(E_{R_\gamma}\left[G_\ell(\alpha,\mu,\omega)\right]-\frac{1}{N_\ell}\sum_{j=1}^{N_\ell}G_\ell(\alpha,\mu,\omega_j)\right),\label{Proof - Thm comp CI -1}
\end{align}
where the last inequality follows from the property of the infimum.  Observe that the term~(\ref{Proof - Thm comp CI -1}) can be expressed as
 $-\sup_{g\in\mathcal{G}_\ell}\frac{1}{\sqrt{N_\ell}}\sum_{j=1}^{N_\ell}\left(g(\omega_j)-E_{R_\gamma}\left[g(\omega)\right]\right).$ Now let $Z_\ell=\sup_{g\in\mathcal{G}_\ell}\frac{1}{\sqrt{N_\ell}}\sum_{j=1}^{N_\ell}\left(g(\omega_j)-E_{R_\gamma}\left[g(\omega)\right]\right)$ for each $\ell$, and by Lemma~\ref{lem:gauss-approx-inf} there exists
 $\tilde{Z}_\ell=\sup_{g\in\mathcal{G}_\ell}B^{l}_\ell g$, where $B^{l}_{\ell}$ is a centered Gaussian process indexed by $\mathcal{G}_\ell$ with covariance function $E_{R_\gamma}\left[B^{l}_{\ell}(g)\,B^{l}_{\ell}(g^\prime)\right]=\int_{\Omega}g(\omega)\,g^\prime(\omega)\,d R_\gamma$, for $g,g^\prime\in\mathcal{G}_\ell$, such that $|Z_\ell-\tilde{Z}_\ell|= O_{R_\gamma}\left(r_\ell\right)$ as $\ell\rightarrow\infty$ with $r_\ell =C\Big[
b K_{N_\ell}\, {N_\ell}^{-1/2}
\;+\; (b\sigma)^{1/2} K_{N_\ell}^{3/4} {N_\ell}^{-1/4}
\;+\; (b \sigma^2 K_{N_\ell}^2)^{1/3} {N_\ell}^{-1/6}
\Big],$
where $K_{N_\ell}= c\,v\max\Big(\log N_\ell, \log\Big(\frac{A b}{\sigma}\Big)\Big)$, with $C,c>0$ are constants \emph{not} depending on $N_\ell$, 
$b=\sigma=e^{\gamma\kappa\|c\|_{\infty}},$ $v=n_\ell+2$, and $A=\left(K(n_\ell+2)\right)^{\frac{1}{{n_\ell+2}}}16e$ with $K$ a universal constant.
  
In consequence, for each $\ell$ we have that $\sqrt{N_\ell}\vartheta^* \geq  \sqrt{N_\ell}\hat{\vartheta}_\ell -\sqrt{N_\ell}\epsilon_\ell\,\gamma\kappa\|c\|_{\infty}e^{\kappa\gamma\|c\|_{\infty}}-Z_\ell+\tilde{Z}_\ell-\tilde{Z}_\ell$, holds; hence,
\begin{align*}
\vartheta^*\geq\hat{\vartheta}_\ell-\epsilon_\ell\,\gamma\kappa\|c\|_{\infty}e^{\kappa\gamma\|c\|_{\infty}}+ O_{R_\gamma}\left(r_\ell/\sqrt{N_\ell}\right)-\tilde{Z}_\ell/\sqrt{N_\ell}\quad \text{as $\ell\rightarrow\infty$}.
\end{align*}
where it is $|\tilde{Z}_\ell-Z_\ell|=O_{R_\gamma}\left(r_\ell/\sqrt{N_\ell}\right)$ as $\ell\rightarrow\infty$. Finally, because $r_\ell\downarrow0$ as $\ell\rightarrow\infty$, we must have $|\tilde{Z}_\ell-Z_\ell|=o_{R_\gamma}\left(1/\sqrt{N_\ell}\right)$ as $\ell\rightarrow\infty$.

\par Now I turn my focus to the upper bound. Starting with $$\sqrt{N_\ell}(\hat{\vartheta}_\ell-\vartheta_\ell)=\sqrt{N_\ell}\left(\inf_{(\alpha,\mu)\in\overline{C}_\ell}\frac{1}{N_\ell}\sum_{j=1}^{N_\ell}G_\ell(\alpha,\mu,\omega_j)-\inf_{(\alpha,\mu)\in\overline{C}_\ell}E_{R_\gamma}\left[G_\ell(\alpha,\mu,\omega)\right]\right),$$
add and subtract $E_{R_\gamma}\left[G_\ell(\alpha,\mu,\omega)\right]$ under the sum in the right side of the above equality and using the lower bound property of the infimum of the sum, I obtain
\begin{align*}
\sqrt{N_\ell}(\hat{\vartheta}_\ell-\vartheta_\ell) & \geq \sqrt{N_\ell}\inf_{(\alpha,\mu)\in\overline{C}_\ell}\frac{1}{N_\ell}\sum_{j=1}^{N_\ell}\left(G_\ell(\alpha,\mu,\omega_j)-E_{R_\gamma}\left[G_\ell(\alpha,\mu,\omega)\right]\right)\\
& \qquad +\sqrt{N_\ell}\inf_{(\alpha,\mu)\in\overline{C}_\ell}E_{R_\gamma}\left[G_\ell(\alpha,\mu,\omega)\right]-\sqrt{N_\ell}\inf_{(\alpha,\mu)\in\overline{C}_\ell}E_{R_\gamma}\left[G_\ell(\alpha,\mu,\omega)\right]\\
 & = \sqrt{N_\ell}\inf_{(\alpha,\mu)\in\overline{C}_\ell}\frac{1}{N_\ell}\sum_{j=1}^{N_\ell}\left(G_\ell(\alpha,\mu,\omega_j)-E_{R_\gamma}\left[G_\ell(\alpha,\mu,\omega)\right]\right).
\end{align*}
Furthermore, $\vartheta^*\leq \vartheta_\ell$ for each $\ell$, and hence, 
\begin{align}\label{Proof - Thm comp CI 0}
\sqrt{N_\ell}(\hat{\vartheta}_\ell-\vartheta^*) & \geq \sqrt{N_\ell}\inf_{(\alpha,\mu)\in\overline{C}_\ell}\frac{1}{N_\ell}\sum_{j=1}^{N_\ell}\left(G_\ell(\alpha,\mu,\omega_j)-E_{R_\gamma}\left[G_\ell(\alpha,\mu,\omega)\right]\right).
\end{align}

\par The right side of~(\ref{Proof - Thm comp CI 0}) is equivalent to $-\sup_{g\in\mathcal{G}^{-}_\ell}\frac{1}{\sqrt{N_\ell}}\sum_{j=1}^{N_\ell}\left(g(\omega_j)-E_{R_\gamma}\left[g(\omega)\right]\right),$
where $\mathcal{G}^{-}_\ell=-\mathcal{G}_\ell$. Now let $Z_\ell=\sup_{g\in\mathcal{G}^{-}_\ell}\frac{1}{\sqrt{N_\ell}}\sum_{j=1}^{N_\ell}\left(g(\omega_j)-E_{R_\gamma}\left[g(\omega)\right]\right)$ for each $\ell$. Because $\mathcal{G}^{-}_\ell$ must also be VC-subgraph by Part (iv) of Lemma 2.6.18 in~\cite{VDV-W}, and that it must also be pointwise measurable, and has the same envelope function as that of $\mathcal{G}_\ell$, I can apply Lemma~\ref{lem:gauss-approx-inf} to deduce the existence of $\tilde{Z}_\ell=\sup_{g\in\mathcal{G}^{-}_\ell}B^{u}_\ell g$, where $B^{u}_{\ell}$ is a centered Gaussian process indexed by $\mathcal{G}^{-}_\ell$ with covariance function $E_{R_\gamma}\left[B^{u}_{\ell}(g)\,B^{u}_{\ell}(g^\prime)\right]=\int_{\Omega}g(\omega)\,g^\prime(\omega)\,d R_\gamma$, for $g,g^\prime\in\mathcal{G}^{-}_\ell$, such that $|Z_\ell-\tilde{Z}_\ell|= O_{R_\gamma}\left(r_\ell\right)$ as $\ell\rightarrow\infty$ with $r_\ell$ as with the lower bound. 

\par In consequence, $\sqrt{N_\ell}(\hat{\vartheta}_\ell-\vartheta^*)  \geq (\tilde{Z}_\ell-Z_\ell)-\tilde{Z}_\ell\implies \vartheta^*\leq |\tilde{Z}_\ell-Z_\ell|/\sqrt{N_\ell}+\tilde{Z}_\ell/\sqrt{N_\ell}+\hat{\vartheta}_\ell$. Now noting that $|\tilde{Z}_\ell-Z_\ell|/\sqrt{N_\ell}=O_{R_\gamma}\left(r_\ell/\sqrt{N_\ell}\right)$ as $\ell\rightarrow\infty$ and because $r_\ell\downarrow0$ as $\ell\rightarrow\infty$, this term is in fact $o_{R_\gamma}\left(1/\sqrt{N_\ell}\right)$ as $\ell\rightarrow\infty$.
\end{proof}

\section{Technical Lemmas for Proposition~\ref{Thm - Fenchel}}\label{Section - Propo-Fenchel}
\begin{lemma}\label{Lemma - H-L}
Suppose the conditions of Proposition~\ref{Thm - Fenchel}, hold. Furthermore, for each $n\in\mathbb{Z}_+$, let $z_{0,n}=\arginf\left\{ \int_{\Omega}e^{z}\,dR_\gamma : y\in\mathcal{D}_n\right\}$ where $\mathcal{D}_n=\left\{y\in\mathcal{D}:\alpha\leq \bar{\alpha}_n\right\}$, with $\bar{\alpha}_n\nearrow\infty$ as $n\rightarrow\infty$. Then $$\text{OL}\left\{\{z_{0,n}\}_{n\geq1}\right\}\subset \arginf\left\{\int_{\Omega}e^{z}\,dR_\gamma: z\in \mathcal{D}\right\}.$$
\end{lemma}
\begin{proof}
I will establish the three conditions of Assumption~2.1 in~\citet{Mena-Lerma-2005} hold in this example. Assumption~2.1(a) requires the set of limit points of the sequence $\{z_{0,n}\}_{n\geq1}$ in the norm topology of $L_1(R_\gamma)$, to be a subset of $\mathcal{D}$. We consider two case: (i) the set of limit points is empty, and (ii) the set of limit points is non-empty. In case (i), Assumption~2.1(a) trivially holds since the empty set is a subset of every set. In case (ii), there exists a subsequence $\{z_{0,n_k}\}_{k\geq1}$ of $\{z_{0,n}\}_{n\geq1}$ that converges to a limit, $z$, and we must show $z\in\mathcal{D}$. Toward that end, first note that $z_{0,n_k}\in\mathcal{D}_{n_k}$ for each $k$ means $z_{0,n_k}=\alpha_{0,n_k} z^\prime_{0,n_k}$ where  $\alpha_{0,n_k}\in[0,\bar{\alpha}_{n_k}]$ and $z^\prime_{0,n_k}\in\overline{\text{co}}(\mathcal{V})$ for each $k$. Now since $z\in L_1(R_\gamma)$, $\exists L>0$ such that $\alpha_{0,n_k}\leq L$ for all $k$, and hence, $\{z_{0,n_k}\}_{k\geq1}\subset \mathcal{D}_{L^\prime}$, where $L^\prime=\min\{n_k:\bar{\alpha}_{n_k}\geq L\}$. Since $\mathcal{D}_{L^\prime}$ is closed in the norm topology of $L_1(R_\gamma)$, we must have $z\in\mathcal{D}_{L^\prime}$, and since $\mathcal{D}_{L^\prime}\subset \mathcal{D}$, it follows that $z\in\mathcal{D}$, as desired.

\par Assumption~2.1(b) requires every subsequence $\{z_{0,n_k}\}_{k\geq1}$ that converges to a limit, $z$, to satisfy $$\liminf_{k\rightarrow\infty}\int_{\Omega}e^{z_{0,n_k}}\,dR_\gamma\geq \int_{\Omega}e^{z}\,dR_\gamma.$$ As with the proof of Assumption~2.1(a),
since $z\in L_1(R_\gamma)$, $\exists L>0$ such that $\alpha_{0,n_k}\leq L$ for all $k$. Then for some $\eta>0$, $\int_{\Omega}e^{z_{0,n_k}}\,dR_\gamma\leq e^{\eta\,L}$ for all $k$, because $|z_{0,n_k}|\leq L$ for each $k$. Hence, the sequence $\{e^{z_{0,n_k}}\}_{k\geq1}$ is uniformly integrable. Consequently, $\lim_{k\rightarrow\infty}\int_{\Omega}e^{z_{0,n_k}}\,dQ=\int_{\Omega}e^{z}\,dR_\gamma,$ holds, implying the desired result.

\par Assumption~2.1(c) requires for each $z\in\mathcal{D}$, $\exists N\in\mathbb{N}$ and sequence $\{z_{n}\}_{n\geq1}$ with $z_n\in\mathcal{D}_n$ for all $n\geq N$, and such that $z_n\stackrel{L_1(R_\gamma)}{\longrightarrow}y$ and $\lim_{n\rightarrow\infty}\int_{\Omega}e^{z_n}\,dR_\gamma= \int_{\Omega}e^{z}\,dR_\gamma$. For $z\in\mathcal{D}$ means $z=\alpha\,z^\prime$, and that $\exists N\in\mathbb{N}$ such that $\bar{\alpha}_n\geq \alpha$ for all $n\geq N$. Set $z_n=\alpha_n\,z^\prime$ such that 
$\alpha_n\leq \bar{\alpha}_n$ and $\lim_{n\rightarrow\infty}\alpha_n=\alpha$. Observe that $z_n\stackrel{L_1(Q)}{\longrightarrow}z$, holds, so that now we can repeat the same arguments in the previous paragraph to deduce that the sequence $\{e^{z_{n}}\}_{n\geq1}$ is uniformly integrable, to conclude
\begin{align*}
\lim_{n\rightarrow\infty}\int_{\Omega}e^{z_n}\,dR_\gamma=\int_{\Omega}e^{z}\,dR_\gamma.
\end{align*}
This concludes the proof.
\end{proof}

\begin{lemma}\label{Lemma - Fenchel Dual existence of cluster point}
Suppose the conditions of Proposition~\ref{Thm - Fenchel}, hold. Then $\text{OL}\left\{\{z_{0,n}\}_{n\geq1}\right\}\neq\emptyset$.
\end{lemma}

\begin{proof}
For each $n$, let $p_{0,n}=e^{z_{0,n}}/\int_{\Omega}e^{z_{0,n}}\,dR_\gamma$. I will now establish that $\{p_{0,n}\}_{n\geq1}\subset \mathcal{M}$. The set $\mathcal{D}_n$ is convex, and the objective function $g(y)=\int_{\Omega}e^{z}\,dR_\gamma$ is G\^{a}teaux differentiable, then by Theorem~2 on page 178 of~\citet{Luenberger}, $\frac{d}{dt}g\left(z_0+t(z-z_0)\right)\mid_{t=0}\,\geq0$ $\forall z\in \mathcal{D}_n$, yielding $\int_{\Omega}(z-z_{0,n})e^{z_{0,n}}\,dR_\gamma\geq0$ $\forall z\in \mathcal{D}_n$. By choosing $z=cz_{0,n}$ first with $c>1$ and then with $c<1$ (since $\mathcal{D}_n$ is also a cone), we obtain $\int_{\Omega}z_{0,n}e^{z_{0,n}}\,dR_\gamma=0$ and $\int_{\Omega}ze^{z_{0,n}}\,dR_\gamma\geq0$ $\forall z\in \mathcal{D}_n$. Since $\alpha_{0,n}, \int_{\Omega}e^{z_{0,n}}\,dR_\gamma>0$, these conditions are equivalent to 
\begin{align}\label{Example 1 - 3}
\int_{\Omega}\alpha^\prime_{0,n}\,p_{0,n}\,dR_\gamma=0\quad\text{and}\quad \int_{\Omega}z^\prime p_{0,n}\,dR_\gamma\geq0\;\forall a\in \mathcal{D}_n
\end{align}
where $z_{0,n}=\alpha_{0,n}\,z^\prime_{0,n}$. Observe that $p_{0,n}$ is clearly a density function, so the main point to show is that it satisfies the moment inequality constraints. To show that it satisfies them, we use the fact that $v\in\overline{\text{co}}(\mathcal{V})$ $\forall v\in\mathcal{V}$. Now setting $z^\prime=v$ and plugging this choice into the second condition in (\ref{Example 1 - 3}), we to obtain $\int_{\Omega}v p_{0,n}\,dR_\gamma\leq0,$ which holds for each $v\in\mathcal{V}$. Therefore, $p_{0,n}\in\mathcal{M}$. As $n$ was arbitrary, $\{p_{0,n}\}_{n\geq1}\subset \mathcal{M}$ must hold.

\par Now I will now establish that $\lim_{n\rightarrow\infty}\alpha_{0,n}=\infty$ \emph{cannot} arise. For suppose that $\lim_{n\rightarrow\infty}\alpha_{0,n}=\infty$, holds, and consider a subsequence 
$\{z^\prime_{0,n_k}\}_{k\geq1}$ such that $z^\prime_{0,n_k}\stackrel{L_1(R_\gamma)}{\longrightarrow} z^\prime_\infty\in\overline{\text{co}}(\mathcal{V})$. The existence of such a subsequence is because $\{z^\prime_{0,n_k}\}_{k\geq1}\subset\overline{\text{co}}(\mathcal{V})$ and $\overline{\text{co}}(\mathcal{V})$ is compact in the norm topology of $L_1(R_\gamma)$ (i.e., $L_1(R_\gamma)$ is a Fr\'{e}chet space and $\mathcal{V}$ is precompact in the norm topology of $L_1\left(R_\gamma\right)$ by Lemma~\ref{Lemma - Precomp}). Furthermore, note that 
$\lim_{n\rightarrow\infty}\int_{\Omega}e^{z_{0,n}}\,dR_\gamma$ exists since $\left\{\int_{\Omega}e^{z_{0,n}}\,dR_\gamma\right\}_{n\geq1}$ is a non-increasing sequence of real numbers that is bounded from below by 1. Now, taking a further subsequence 
$\{z^\prime_{0,n_{k_\ell}}\}_{\ell\geq1}$ where $z^\prime_{0,n_{k_\ell}}\rightarrow z^\prime_\infty$ a.s.$-R_\gamma$, 
\begin{align*}
\int_{\Omega}e^{z_{0,n_{k_\ell}}}\,dR_\gamma & = \int_{\Omega}e^{z_{0,n_{k_\ell}}}\,\left(1\left[e^{z_{0,n_{k_\ell}}}>\alpha_{0,n_{k_\ell}}\right]+1\left[e^{z_{0,n_{k_\ell}}}\leq\alpha_{0,n_{k_\ell}}\right]\right)\,dR_\gamma\\
& \geq \alpha_{0,n_{k_\ell}}\,R_\gamma\left(\omega\in\Omega:e^{z_{0,n_{k_\ell}}}>\alpha_{0,n_{k_\ell}} \right)\\
& = \alpha_{0,n_{k_\ell}}\,R_\gamma\left(\omega\in\Omega:z^\prime_{0,n_{k_\ell}}>\frac{\log\alpha_{0,n_{k_\ell}}}{\alpha_{0,n_{k_\ell}}}\right).
\end{align*} 
Now taking limits as $\ell\rightarrow\infty$ on both sides of these inequalities implies $$\lim_{\ell\rightarrow\infty}R_\gamma\left(\omega\in\Omega:z^\prime_{0,n_{k_\ell}}>\frac{\log\alpha_{0,n_{k_\ell}}}{\alpha_{0,n_{k_\ell}}}\right)=0,$$ since we are assuming
$\lim_{n\rightarrow\infty}\alpha_{0,n}=\infty$, which implies $\lim_{\ell\rightarrow\infty}\alpha_{0,n_{k_\ell}}=\infty$. And by the Bounded Convergence Theorem,
\begin{align*}
0=\lim_{\ell\rightarrow\infty}R_\gamma\left(\omega\in\Omega:z^\prime_{0,n_{k_\ell}}>\frac{\log\alpha_{0,n_{k_\ell}}}{\alpha_{0,n_{k_\ell}}}\right)& =\int_{\Omega}\lim_{\ell\rightarrow\infty}1\left[\omega\in\Omega:z^\prime_{0,n_{k_\ell}}>\frac{\log\alpha_{0,n_{k_\ell}}}{\alpha_{0,n_{k_\ell}}}\right]\,dR_\gamma\\
& = \int_{\Omega}1\left[\omega\in\Omega:y^\prime_\infty>0\right]\,dQ=Q\left(\omega\in\Omega:y^\prime_\infty>0\right).
\end{align*}
However, $z^\prime_\infty\in\overline{\text{co}}(\mathcal{V})$ means there exists $\{v_{\gamma_i},i=1,\ldots,m\}\subset\overline{\mathcal{V}}=\mathcal{V}$ and $\lambda_i>0$ for $i=1,\ldots m$ with $\sum_{i=1}^{m}\lambda_i=1$, such that  $z^\prime_\infty=\sum_{i=1}^{m}\lambda_i\,v_{\gamma_i}$. In consequence, 
$$0=R_\gamma\left(\omega\in\Omega:z^\prime_\infty>0\right)\geq (1-\max_i \gamma_i)>0,$$ yielding a contradiction. 

\par Therefore, $\lim_{n\rightarrow\infty}\int_{\Omega}e^{z_{0,n}}\,dR_\gamma>0$, holds. Note that this limit is finite, since $\int_{\Omega}e^{z_{0,n}}\,dR_\gamma=e^{-m(p_{0,n})}\leq e^{-m(p_\gamma)}<\infty$ for every $n$, where $p_\gamma$ is the $I$-projection onto $\mathcal{M}$. Hence, there exists a subsequence $\{\alpha_{0,n_k}\}_{k\geq1}$ such that $\lim_{k\rightarrow\infty}\alpha_{0,n_k}\in\mathbb{R}_{++}$ (i.e., positive real numbers), and denote this limit by $\alpha_0$. Now taking a further subsequence, $\{z^\prime_{0,n_{k_\ell}}\}_{\ell\geq1}$, such that $z^\prime_{0,n_{k_\ell}}\stackrel{L_1(R_\gamma)}{\longrightarrow} z^\prime_\infty\in\overline{\text{co}}(\mathcal{V})$, observe that 
$z_{0,n_{k_\ell}}\stackrel{L_1(R_\gamma)}{\longrightarrow} \alpha_0z^\prime_\infty\in\text{OL}\left\{\{z_{0,n}\}_{n\geq1}\right\}$. This concludes the proof.
\end{proof}

\begin{lemma}\label{Lemma - Precomp}
Suppose Part (i) in Assumption~\ref{Assump -Primities OTP} holds. For each $\gamma\in\mathbb{Z}_+$, the set $\mathcal{V}$ is precompact in the norm topology of $L_1\left(R_\gamma\right)$.
\end{lemma}
\begin{proof}
The proof proceeds by the direct method. The set $\mathcal{V}$ can be expressed as the union of 4 sets, In particular, $\mathcal{V}=\bigcup_{i=1}^4\mathcal{V}_i$, where
\begin{align*}
\mathcal{V}_1 & =\left\{\left(F_X(x^\prime)-1\left[x\preceq x^\prime\right]\right)1\left[y\in\mathcal{Y}\right],x^\prime\in\mathcal{X}\right\},\quad\mathcal{V}_2=-\mathcal{V}_1 \\
\mathcal{V}_3 & =\left\{\left(F_Y(y^\prime)-1\left[y\preceq y^\prime\right]\right)1\left[x\in\mathcal{X}\right],y^\prime\in\mathcal{Y}\right\},\quad \mathcal{V}_4 =-\mathcal{V}_3.
\end{align*}
As norm precompactness in $L_1\left(R_\gamma\right)$ is preserved under finite unions of such sets, the proof requires that we establish these two properties for each of the aforementioned sets. Now because $\mathcal{V}_2=-\mathcal{V}_1$ and $\mathcal{V}_4=-\mathcal{V}_3$, and that $\mathcal{V}_1$ and $\mathcal{V}_3$ consist of the same type of elements with only the roles of the two marginals being reversed, it is sufficient to only show that $\mathcal{V}_1$ is precompact to obtain the desired result. Note that identical arguments would apply to establishing $\mathcal{V}_3$ being precompact, and the remaining sets are negatives of the previous two, which preserves the two properties.

\par I now establish that $\mathcal{V}_1$ is precompact in the norm topology of $L_1\left(R_\gamma\right)$. We know that the collection of sets $\left\{\{x\preceq x^{\prime}\},:x^{\prime}\in\mathcal{X}\right\}$ is a VC-class with index $d_x+1$ -- see, e.g., Example 2.6.1 in~\cite{VDV-W}. This implies that the class of indicator functions $\left\{1\left[x\preceq x^{\prime}\right]:x^{\prime}\in\mathcal{X}\right\}$ is VC-subgraph
in $\mathcal{X}\times \mathbb{R}$. By Part (iv) of Lemma~2.6.18 in~\cite{VDV-W}, the collection $\left\{-1\left[x\preceq x^{\prime}\right]:x^{\prime}\in\mathcal{X}\right\}$ is also VC-subgraph
in $\mathcal{X}\times \mathbb{R}$. Now since the map $$-1\left[x\preceq x^{\prime}\right]\mapsto F(x^\prime)-1\left[x\preceq x^{\prime}\right]$$ is monotone, Part (viii) of Lemma~2.6.18 in~\cite{VDV-W} establishes that
the collection $\mathcal{V}_1$ is also VC-subgraph in $\mathcal{X}\times \mathbb{R}$. Consequently, for every $\epsilon>0$, Theorem~2.6.7 in~\cite{VDV-W} establishes the covering number of $\mathcal{V}_1$ in $L_1(R_\gamma)$ is finite, and hence, $\mathcal{V}_1$ is appropriately precompact.
\end{proof}

\begin{lemma}\label{Lemma - Closed}
Suppose Part (i) in Assumption~\ref{Assump -Primities OTP} holds. Then $\overline{\mathcal{V}}=\left\{g_\ell,-g_\ell:\ell\in\mathcal{X}\cup\mathcal{Y}\right\}$, where closure is in the norm topology of $L_1\left(R_\gamma\right)$.
\end{lemma}
\begin{proof}
The proof proceeds by the direct method and uses the same representation of $\mathcal{V}$ as in Lemma~\ref{Lemma - Precomp}. Starting with $\mathcal{V}=\bigcup_{i=1}^4\mathcal{V}_i$, where
\begin{align*}
\mathcal{V}_1 & =\left\{\left(F_X(x^\prime)-1\left[x\preceq x^\prime\right]\right)1\left[y\in\mathcal{Y}\right],x^\prime\in\mathcal{X}\right\},\quad\mathcal{V}_2=-\mathcal{V}_1 \\
\mathcal{V}_3 & =\left\{\left(F_Y(y^\prime)-1\left[y\preceq y^\prime\right]\right)1\left[x\in\mathcal{X}\right],y^\prime\in\mathcal{Y}\right\},\quad \mathcal{V}_4 =-\mathcal{V}_3, 
\end{align*}
as norm closedness in $L_1\left(R_\gamma\right)$ is preserved under finite unions of such sets, the proof requires that we establish these two properties for each of the aforementioned sets. Now because $\mathcal{V}_2=-\mathcal{V}_1$ and $\mathcal{V}_4=-\mathcal{V}_3$, and that $\mathcal{V}_1$ and $\mathcal{V}_3$ consist of the same type of elements with only the roles of the two marginals being reversed, it is sufficient to only show that $\mathcal{V}_1$ is closed to obtain the desired result. Note that identical arguments would apply to establishing $\mathcal{V}_3$ is closed, and the remaining sets are negatives of the previous two, which preserves the two properties.

\par I now establish that the limit points of $\mathcal{V}_1$, in the norm topology of $L_1\left(R_\gamma\right)$, have the same form of its elements. Consider an arbitrary sequence $\{v_n\}_{n\geq1}\subset \mathcal{G}_1$, where
$$v_n(x,y)=\left(F_X(x_n^\prime)-1\left[x\preceq x_n^\prime\right]\right)1\left[y\in\mathcal{Y}\right]\quad n=1,2,\ldots,$$
 such that $v_n\stackrel{L_1\left(R_\gamma\right)}{\longrightarrow} v$. To prove the desired result, we need to establish that $$v(x,y)=\left(F_{X}(x^{\prime})-1\left[x\preceq x^{\prime}\right]\right)\,1[y\in\mathcal{Y}]$$ for some $x^\prime\in\mathcal{X}$. As norm-convergence in $L_1\left(R_\gamma\right)$ implies convergence in $R_\gamma$-measure, there exists a non-random increasing sequence of integers $n_1,n_2,\ldots,$ such that $\{v_{n_k}\}_{k\geq1}$ converges to $v$ a.s.$-R_\gamma$. That is,
\begin{align}\label{Proof - precompact 0}
v(x,y)=\lim_{k\rightarrow+\infty} v_{n_k}(x,y)\quad \text{a.s.}-R_\gamma.
\end{align}
The limit~(\ref{Proof - precompact 0}) implies $\{x^\prime_{n_k}\}_{k\geq0}\subset \mathcal{X}$ holds. Since $\mathcal{X}\subset\mathbb{R}^{d_x}$ is compact, there exists a subsequence $\{x^\prime_{n_{k_s}}\}_{s\geq1}$ such that $\lim_{s\rightarrow +\infty}x^\prime_{n_{k_s}}=x^{\prime,\star}\in\mathcal{X}$. If $F_X$ is continuous at $x^{\prime,\star}$, then combining this conclusion with the limit~(\ref{Proof - precompact 0}) yields
\begin{align*}
v(x,y)=\lim_{s\rightarrow+\infty} v_{n_{k_s}}(x,y)=\left(F_{X}(x^{\prime,\star})-1\left[x\preceq x^{\prime,\star}\right]\right)\,1[y\in\mathcal{Y}]\quad \text{a.s.}-R_\gamma,
\end{align*}
as every subsequence of $\{v_{n_k}\}_{k\geq1}$ converges to $v$ a.s.$-R_\gamma$. Therefore, $v$ has the appropriate form. 

\par Now, we focus on the case where $F_X$ is discontinuous at $x^{\prime,\star}$. The CDF $F_X$ is monotonic, so it can have at most a countable number of points of discontinuity. With $F_X$ discontinuous at $x^{\prime,\star}$, by following steps identical those above for the continuous case, we have
\begin{align*}
v(x,y)=\lim_{s\rightarrow+\infty} v_{n_{k_s}}(x,y)=\left(F_{X}(x^{\prime,\star}-)-1\left[x\preceq x^{\prime,\star}\right]\right)\,1[y\in\mathcal{Y}]\quad \text{a.s.}-R_\gamma,
\end{align*}
where $\lim_{s\rightarrow+\infty}F_X\left(x^\prime_{n_{k_s}}\right)=F_X\left(x^{\prime,\star}-\right)$ and $x^\prime_{i,n_{k_s}}\uparrow x^{\prime,\star}_i$ for some component $i$ at which $F_X$ is discontinuous.
\end{proof}

\begin{lemma}\label{Lemma - KMT and MT}
Suppose Part (i) of Assumption~\ref{Assump -Primities OTP} holds. The following statements hold: $\text{ex}\left(\overline{\text{co}}(\mathcal{V})\right)\neq\emptyset$, $\overline{\text{co}}\left(\text{ex}\left(\overline{\text{co}}(\mathcal{V})\right)\right)=\overline{\text{co}}(\mathcal{V})$, and
$\text{ex}\left(\overline{\text{co}}(\mathcal{V})\right)\subset\mathcal{V}$.
\end{lemma}
\begin{proof}
The proof proceeds by the direct method. By Lemmas~\ref{Lemma - Precomp} and~\ref{Lemma - Closed}, the set $\mathcal{V}$ is compact in the $L_1(R_\gamma)$-norm (as it is complete and totally bounded). Now, since $L_1(R_\gamma)$ is a Banach space, it is therefore a Fr\'{e}chet space. This fact allows us to apply part (c) of Theorem~3.20 in~\citet{Rudin-Book} to the set $\mathcal{V}$ to deduce that $\overline{\text{co}}(\mathcal{V})$ is also compact in the same norm. Whence, I can apply the Krein-Milman Theorem (e.g., Theorem~3.23 in~\citealp{Rudin-Book}) to the set $\overline{\text{co}}(\mathcal{V})$ to deduce that its set of extreme points, $\text{ex}\left(\overline{\text{co}}(\mathcal{V})\right)$, is nonempty, and that
$\overline{\text{co}}\left(\text{ex}\left(\overline{\text{co}}(\mathcal{V})\right)\right)=\overline{\text{co}}(\mathcal{V})$. Next, I can apply Milman's Theorem (e.g., Theorem~3.25 in~\citealp{Rudin-Book}) to establish $\text{ex}\left(\overline{\text{co}}(\mathcal{V})\right)\subset\mathcal{V}$.
\end{proof}

\section{Technical Results for Theorem~\ref{Thm Stochastic Opt Sieve}}\label{Appendix - ULLN}
Consider a class of functions $\mathcal{F}$ of real-valued measurable functions defined on the probability space $(\mathcal{Z},\mathcal{A},P)$. The sample is denoted by $Z_1,\ldots,Z_N$. Furthermore, let $\varepsilon_1,\ldots,\varepsilon_N$ denotes a random sample of Rademacher variables that is independent $Z_1,\ldots,Z_N$. The Rademacher complexity of the function class $\mathcal{F}$ is defined as
\begin{align}
\mathcal{R}_N(\mathcal{F})=E_{\left(P_Z\otimes P_\varepsilon\right)^{\otimes_N}}\left[\sup_{f\in\mathcal{F}}\left|\frac{1}{N}\sum_{i=1}^N\varepsilon_i f(Z_i)\right|\right].
\end{align}
\begin{proposition}\label{Prop - ULLN}
Consider a sequence $\{\mathcal{F}_\ell\}_{\ell\geq 1}$ where $\mathcal{F}_\ell$ is a class of real-valued measurable functions defined on the probability space $(\mathcal{Z},\mathcal{A},P)$, such that $\|f\|_{L_\infty(P)}\leq b$ $\forall f\in\mathcal{F}_\ell$ for each $\ell$. Furthermore, let $\{{N_\ell}\}_{\ell\geq1}$ be a sequence of sample sizes such that $\lim_{\ell\rightarrow\infty}N_\ell=\infty$. If $\mathcal{R}_{N_\ell}(\mathcal{F}_\ell)=o(1)$ as $\ell\rightarrow\infty$, then $\sup_{f\in\mathcal{F}_\ell}\left|\frac{1}{N_\ell}\sum_{i=1}^{N_\ell}f(Z_i)-E(f(Z_i))\right|\stackrel{\text{a.s.}}{\longrightarrow}0$.
\end{proposition}
\begin{proof}
Fix $\epsilon>0$ and define
\[
A_{N_\ell}(\epsilon)
:=
\left\{
\sup_{f\in\mathcal{F}_\ell}
\left|
\frac{1}{N_\ell}\sum_{i=1}^{N_\ell} f(Z_i)
-
E f(Z_1)
\right|
>
\epsilon
\right\},
\qquad
B_m(\epsilon)
:=
\bigcup_{N_\ell \ge m} A_{N_\ell}(\epsilon).
\]
It suffices to show that
$
P(B_m(\epsilon)) \to 0
$
as $m\to\infty$, since this implies
\[
\sup_{f\in\mathcal{F}_\ell}
\left|
\frac{1}{N_\ell}\sum_{i=1}^{N_\ell} f(Z_i)
-
E f(Z_1)
\right|
\to 0
\quad\text{almost surely}.
\]

By assumption, $\mathcal{R}_{N_\ell}(\mathcal{F}_\ell)=o(1)$.
Hence there exists $\ell_0$ such that for all $\ell\ge \ell_0$,
\[
2\,\mathcal{R}_{N_\ell}(\mathcal{F}_\ell) < \epsilon/2.
\]
Let $\delta := \epsilon/2$.
Then, for all $\ell\ge \ell_0$,
\[
A_{N_\ell}(\epsilon)
\subseteq
\left\{
\sup_{f\in\mathcal{F}_\ell}
\left|
\frac{1}{N_\ell}\sum_{i=1}^{N_\ell} f(Z_i)
-
E f(Z_1)
\right|
>
2\mathcal{R}_{N_\ell}(\mathcal{F}_\ell)+\delta
\right\}.
\]

By Theorem~4.10 of \cite{Wainwright_2019}, since
$\|f\|_{L_\infty(P)}\le b$ for all $f\in\mathcal{F}_\ell$,
\[
P\!\left(
\sup_{f\in\mathcal{F}_\ell}
\left|
\frac{1}{N_\ell}\sum_{i=1}^{N_\ell} f(Z_i)
-
E f(Z_1)
\right|
>
2\mathcal{R}_{N_\ell}(\mathcal{F}_\ell)+\delta
\right)
\le
\exp\!\left(-\frac{N_\ell \delta^2}{2b^2}\right).
\]

Therefore, for all $m$ sufficiently large,
\[
P(B_m(\epsilon))
\le
\sum_{N_\ell\ge m}
\exp\!\left(-\frac{N_\ell \delta^2}{2b^2}\right).
\]

Since $N_\ell\to\infty$, the right-hand side is the tail of a
convergent geometric series and hence converges to zero as
$m\to\infty$.
This proves the claim.
\end{proof}

\begin{lemma}[Expected Rademacher complexity of $\exp$ of convex mixtures]\label{Lemma - Rademacher complexity G}
Suppose Assumption~\ref{Assump -Primities OTP} holds. Then for each sample size $N_\ell$, the collection $\mathcal{G}_\ell$ in~(\ref{eq - Collection VC}) satisfies
\[
\mathcal{R}_{N_\ell}(\mathcal G_\ell)
\;\le\; \gamma\kappa\|c\|_{\infty}\,e^{\gamma\kappa\|c\|_{\infty}}\ \mathcal{R}_{N_\ell}\!\Big(\operatorname{conv}\{v_1,\ldots,v_{n_\ell}\}\Big)
\;=\;\gamma\kappa\|c\|_{\infty}\,e^{\gamma\kappa\|c\|_{\infty}}\mathcal{R}_{N_\ell}\!\big(\{v_1,\ldots,v_{n_\ell}\}\big).
\]
\end{lemma}

\begin{proof}
The proof proceeds by the direct method. To ease exposition, we introduce the following notation: 
$b=\gamma\kappa\|c\|_{\infty}$ and $\Delta_\ell=\{\mu\in\mathbb R_+^n:\sum_{i=1}^{n_\ell} \mu_i=1\}$. Furthermore, fix a sample $S=(\omega_1,\ldots,\omega_{N_\ell})$ and write $s_\mu:=\sum_{i=1}^{n_\ell} \mu_i v_i$.
Since $v_i\in[-1,1]$ for each $i$ and $\mu\in\Delta_\ell$, $s_\mu\in[-1,1]$, hence for $\alpha\in[0,b]$, $\alpha s_\mu\in[-b,b]$.

\emph{(i) Scaling.} Let $\mathbb S:=\{s_\mu:\mu\in\Delta_\ell\}$ and $\mathcal A:=\{\alpha s:\alpha\in[0,b],\,s\in\mathbb S\}$.
For this fixed $S$,
\[
\mathfrak R_S(\mathcal A)
=\mathbb E_{P_\epsilon^{\otimes_{N_\ell}}}\Big[\sup_{\alpha\in[0,b],\,s\in\mathbb S}\frac{\alpha}{N_\ell}\sum_{j=1}^{N_\ell}\epsilon_j s(\omega_j)\Big]
=b\,\mathfrak R_S(\mathbb S).
\]

\emph{(ii) Contraction.} Define $\phi(u):=e^{u}-1$.
On $[-b,b]$, $\phi$ is $e^{b}$-Lipschitz and $\phi(0)=0$.
Talagrand's contraction inequality (\citealp{LedouxTalagrand1991}, Theorem~4.12) yields
\[
\mathfrak R_S\big(\{\,\phi\circ h: h\in\mathcal A\,\}\big)\ \le\ e^{b}\,\mathfrak R_S(\mathcal A).
\]
Since adding a constant does not change $\mathfrak R_S$, we have
\[
\mathfrak R_S(\mathcal G_\ell)
=\mathfrak R_S\big(\{e^{\alpha s}\}\big)
=\mathfrak R_S\big(\{e^{\alpha s}-1\}\big)
\le e^{b}\,\mathfrak R_S(\mathcal A)
= b\,e^{b}\,\mathfrak R_S(\mathbb S).
\]

\emph{(iii) Convex hull invariance.} For fixed $S$,
\[
\sup_{s\in\operatorname{conv}\{v_i\}}\sum_{j=1}^{N_\ell} \epsilon_j s(\omega_j)
=\max_{1\le i\le n_\ell}\sum_{j=1}^{N_\ell} \epsilon_j v_i(\omega_j),
\]
so $\mathfrak R_S(\mathbb S)=\mathfrak R_S(\{v_1,\ldots,v_{n_\ell}\})$.
Combining (i)-(iii) gives, pointwise in $S$,
\[
\mathfrak R_S(\mathcal G_\ell)\ \le\ b e^{b}\, \mathfrak R_S(\{v_1,\ldots,v_{n_\ell}\}).
\]
Taking expectation over $S\sim P^{\otimes_{N_\ell}}$ yields the desired inequality for $\mathcal{R}_{N_\ell}(\mathcal G_\ell)$.
\end{proof}

\begin{corollary}[Finite dictionary]\label{Corollary - Rademacher complexity}
For any $N_\ell\ge1$,
\[
\mathcal R_{N_\ell}\!\big(\{v_1,\ldots,v_{n_\ell}\}\big)
\ \le\ \sqrt{\frac{2\log n_\ell}{N_\ell}}.
\]
\end{corollary}
\begin{proof}
Let $\mathcal V_\ell=\{v_1,\dots,v_{n_\ell}\}$ be a finite class of measurable
functions satisfying $|v_i(\omega)|\le 1$ for all $\omega$ and all
$i=1,\dots,n_\ell$.
Recall that the (expected) Rademacher complexity is
\[
\mathcal R_{N_\ell}(\mathcal V_\ell)
=
E_{\left(R_\gamma\otimes P_\sigma\right)^{\otimes_{N_\ell}}}\!\left[
\sup_{v\in\mathcal V_\ell}
\frac{1}{N_\ell}\sum_{j=1}^{N_\ell}\sigma_j v(\omega_j)
\right],
\]
where $S=(\omega_1,\dots,\omega_{N_\ell})$ is an i.i.d.\ sample and
$\sigma_1,\dots,\sigma_{N_\ell}$ are independent Rademacher variables.

Fix an arbitrary sample $S$.  For each $v_i\in\mathcal V_\ell$, define the
vector
\[
z_i
:=
\big(v_i(\omega_1),\dots,v_i(\omega_{N_\ell})\big)
\in\mathbb R^{N_\ell},
\]
and let
\[
A_S := \{z_1,\dots,z_{n_\ell}\}\subset\mathbb R^{N_\ell}.
\]
Then
\[
\sup_{v\in\mathcal V_\ell}
\frac{1}{N_\ell}\sum_{j=1}^{N_\ell}\sigma_j v(\omega_j)
=
\frac{1}{N_\ell}
\sup_{z\in A_S}\sum_{j=1}^{N_\ell}\sigma_j z_j .
\]

Since $|v_i(\omega)|\le 1$, we have
\[
\|z_i\|_2^2
=
\sum_{j=1}^{N_\ell} v_i(\omega_j)^2
\le
N_\ell,
\qquad i=1,\dots,n_\ell,
\]
and therefore
\[
r_S := \max_{z\in A_S}\|z\|_2 \le \sqrt{N_\ell}.
\]

Applying Massart’s lemma (Theorem~3.3 in~\citealp{MohriRostamizadehTalwalkar2018})
conditionally on $S$ yields
\[
E_{P_\sigma^{\otimes_{N_\ell}}}\!\left[
\frac{1}{N_\ell}
\sup_{z\in A_S}\sum_{j=1}^{N_\ell}\sigma_j z_j
\right]
\le
\frac{r_S}{N_\ell}\sqrt{2\log|A_S|}
\le
\sqrt{\frac{2\log(n_\ell)}{N_\ell}} .
\]
Finally, taking expectation with respect to $S$ proves that
\[
\mathcal R_{N_\ell}(\mathcal V_\ell)
\le
\sqrt{\frac{2\log(n_\ell)}{N_\ell}},
\]
as claimed.
\end{proof}


\begin{lemma}\label{Lemma - LB on Dual value}
Suppose that Assumption~1 holds, and let $\bar{c}=\int_\Omega c(\omega)\,\,d(P_X\otimes P_Y)$. Then $\vartheta^*> e^{-(\gamma \bar{c}+\log a_\gamma)}$.
\end{lemma}
\begin{proof}
The proof proceeds by the direct method. Under Assumption~1, I can apply Theorem~2.2 in~\cite{csiszar1975} in our context with respect to the product measure $P_X\otimes P_Y$ to obtain
$H(P_\gamma|R_\gamma)\leq H(P_X\otimes P_Y|R_\gamma)- H(P_X\otimes P_Y|P_\gamma)$. Consequently, $H(P_\gamma|R_\gamma)< H(P_X\otimes P_Y|R_\gamma)$, holds, since $H(P_X\otimes P_Y|P_\gamma)>0$.

\par Noting that $H(P_X\otimes P_Y|R_\gamma)=\int_{\Omega}\log\left[\frac{dR_\gamma}{d(P_X\otimes P_Y)}(\omega)\right]^{-1}d(P_X\otimes P_Y)=\gamma \bar{c}+\log a_\gamma$, 
the above inequality becomes $H(P_\gamma|R_\gamma)< \gamma \bar{c}+\log a_\gamma$. Finally, I can combine this result with the fact that $H(P_\gamma|R_\gamma)=-\log(\vartheta^*)$ holds to 
deduce the desired result: $\vartheta^*> e^{-(\gamma \bar{c}+\log a_\gamma)}$.
\end{proof}

\begin{lemma}[VC-Dimension]\label{Lemma VC class}
Suppose that Assumption~\ref{Assump -Primities OTP} holds. The collection $\mathcal{G}_\ell$ in~(\ref{eq - Collection VC}) is a Vapnik-Chervonenkis (VC) class of functions with VC dimension $V_\ell$ bounded by $n_\ell+2$, for each $\ell$.
\end{lemma}
\begin{proof}
The proof proceeds by the direct method. For each $n_\ell\in\mathbb{Z}_+$, Lemma~2.6.15 in~\cite{VDV-W} establishes $\text{span}\left(v_1,\ldots,v_{n_\ell}\right)$ is a VC class of functions with VC-dimension bounded above by $n_\ell+2$. Now because $$\left\{\alpha\sum_{i=1}^{n_\ell}\mu_i\,v_i:(\alpha,\mu)\in\mathcal{C}_\ell\right\}\subset \text{span}\left(v_1,\ldots,v_{n_\ell}\right),$$ the class $\left\{\alpha\sum_{i=1}^{n_\ell}\mu_i\,v_i:(\alpha,\mu)\in\mathcal{C}_\ell\right\}$ must also be VC whose VC-dimension is also bounded from above by $n_\ell+2$. I can apply Part (viii) of Lemma~2.6.18 in~\cite{VDV-W} to this last class with exponential function, which is strictly monotonic, to determine that the resulting class is also VC. Since the exponential function is strictly monotonic, and hence, one-to-one, the VC-dimension of this transformed class has the same VC-dimension. Finally, because $\mathcal{G}_\ell\subset \left\{\alpha\sum_{i=1}^{n_\ell}\mu_i\,v_i:(\alpha,\mu)\in\mathcal{C}_\ell\right\}$, it must be that $V_\ell\leq n_\ell+2$. This concludes the proof.
\end{proof}

\begin{lemma}[Gaussian Approximation]\label{lem:gauss-approx-inf}
Let $\mathcal{G}_\ell$ be defined as in~(\ref{eq - Collection VC}), and let $\omega_1,\dots,\omega_{N_\ell}$ be i.i.d.\ with law $R_\gamma$, and define the empirical process
\[
\mathbb{G}_{N_\ell} g
:= \frac{1}{\sqrt{{N_\ell}}} \sum_{i=1}^{N_\ell} \big(g(\omega_i) - E_{R_\gamma}\left[g(\omega)\right]\big),
\qquad g\in\mathcal{G}_{\ell},
\]
Consider the supremum statistic
$Z_\ell=\sup_{g\in\mathcal{G}_{\ell}} \mathbb{G}_{N_\ell} g$. Let $B_\ell$ be a centered tight Gaussian process indexed by $\mathcal{G}_{\ell}$ with covariance function $E_{R_\gamma}\left[B_{\ell}(g)\,B_{\ell}(g^\prime)\right]=\int_{\Omega}g(\omega)\,g^\prime(\omega)\,d R_\gamma$, for $g,g^\prime\in\mathcal{G}_\ell$. Then there exists a random variable
$\tilde Z_{N_\ell} \stackrel{d}{=} \sup_{g\in\mathcal{G}_{\ell}} B_\ell(g)$
such that $|Z_\ell - \tilde Z_\ell|= O_{R_\gamma}\left(r_\ell\right)$,
where
\[
r_\ell
:= C\Big[
b K_{N_\ell}\, {N_\ell}^{-1/2}
\;+\; (b\sigma)^{1/2} K_{N_\ell}^{3/4} {N_\ell}^{-1/4}
\;+\; (b \sigma^2 K_{N_\ell}^2)^{1/3} {N_\ell}^{-1/6}
\Big],
\]
\[
K_{N_\ell} := c\,v\max\Big(\log N_\ell, \log\Big(\frac{A b}{\sigma}\Big)\Big),
\]
with $C,c>0$ are constants \emph{not} depending on $N_\ell$, 
$b=\sigma=e^{\gamma\kappa\|c\|_{\infty}},$ $v=n_\ell+2$, and $A=\left(K(n_\ell+2)\right)^{\frac{1}{{n_\ell+2}}}16e$ with $K$ a universal constant.
\end{lemma}
\begin{proof}
The result follows by applying Corollary~2.2 in~\cite{chernozhukov2014gaussian} in our setup. Towards that end, I shall verify the conditions for its application.

\par The class of functions $\mathcal{G}_\ell$ is trivially pointwise measurable, as it consists of functions that are indicator functions on upper sets plus a constant. By Lemma~\ref{Lemma VC class}, this class is also VC with VC-dimension bounded from above by $n_\ell+2$, and it has the constant function, $\omega\mapsto e^{\kappa\gamma \|c\|_{\infty}}$, as its envelope function. This class of functions is also compatible with Definition~2.1 in \cite{chernozhukov2014gaussian}: by Theorem~2.6.7 in~\cite{VDV-W} and Lemma~\ref{Lemma VC class}, I can specify $v$ and $A$ in their definition as $v=n_\ell+2$ and $A=\left(K(n_\ell+2)\right)^{\frac{1}{{n_\ell+2}}}16e$ with $K$ a universal constant. Since the envelope function is a constant function, I can also specify $b$ and $\sigma$ in their corollary as $b=\sigma=e^{\gamma\kappa\|c\|_{\infty}}$. Finally, I can set $q=\infty$ in their corollary to obtain the desired result because $\mathcal{G}_\ell$ is uniformly bounded.
\end{proof}

\section{Partition of $\mathcal{V}$ Independent of $R_\gamma$}\label{Section - Partition of V}
\begin{lemma}\label{lem:V-partition-epsilon-net}
Recall that $\mathcal{V}$ is a uniformly bounded VC-subgraph class of measurable
functions on the measurable space $\left(\Omega,\mathcal{B}\left(\mathcal{X}\times\mathcal{Y}\right)\right)$, with envelope function, $F(\omega)=1$ for all $\omega\in\Omega$. There exist constants
$C,p > 0$, depending only on the VC characteristics of $\mathcal{V}$, such that the following holds.

For every $\epsilon_\ell \in (0,1)$ and every probability measure $Q$ on
$\left(\Omega,\mathcal{B}\left(\mathcal{X}\times\mathcal{Y}\right)\right)$ there exists a finite subset
$\{v_{1,\ell},\dots,v_{n_\ell,\ell}\} \subset \mathcal{V}$ with
\begin{equation}\label{eq:V-epsilon-net}
  \sup_{v \in \mathcal{V}} \min_{1 \leq i \leq n_\ell}
  \|v - v_{i,\ell}\|_{L_1(Q)} \;\leq\; \epsilon_\ell
\end{equation}
and
\begin{equation}\label{eq:V-nl-polynomial}
  n_\ell \;\leq\; C\,\epsilon_\ell^{-p}.
\end{equation}
Moreover, defining
\[
  E_{i,\ell}
  \;\coloneqq\;
  \Bigl\{ v \in \mathcal{V} :
    \|v - v_{i,\ell}\|_{L_1(Q)}
    = \min_{1 \leq j \leq n_\ell} \|v - v_{j,\ell}\|_{L_1(Q)}
  \Bigr\},
  \qquad i=1,\dots,n_\ell,
\]
with any deterministic tie-breaking rule when the minimum is attained at
multiple indices, yields a partition $\{E_{i,\ell}\}_{i=1}^{n_\ell}$ of
$\mathcal{V}$ such that
\begin{equation}\label{eq:partition-approx-property}
  \sup_{v \in E_{i,\ell}}
  \|v - v_{i,\ell}\|_{L_1(Q)} \;\leq\; \epsilon_\ell,
  \qquad i=1,\dots,n_\ell.
\end{equation}

In particular, these conclusions hold uniformly in the choice of the reference
measure $Q = R_\gamma$ for any $\gamma > 0$.
\end{lemma}

\begin{proof}
Since $\mathcal{V}$ is a uniformly bounded VC--subgraph class with envelope
$F$ satisfying $\|F\|_\infty \leq 1$, standard entropy bounds for VC classes
(see, e.g., Theorem 2.6.7 in \citealp{VDV-W}) imply that there exist constants $A,B > 0$ and an
integer $v \geq 1$ (the VC index) such that, for all probability measures $Q$,
all $\delta \in (0,1)$ and all $r \in (0,1]$,
\begin{equation}\label{eq:VC-entropy-L2}
  \log N\bigl(\delta\,r,\mathcal{V},L_2(Q)\bigr)
  \;\leq\;
  A\,v\,\log\Bigl(\frac{B\,r}{\delta}\Bigr).
\end{equation}
In particular, taking $r = \|F\|_{L_2(Q)} \leq 1$, I obtain
\begin{equation}\label{eq:VC-entropy-L2-simplified}
  \log N\bigl(\delta,\mathcal{V},L_2(Q)\bigr)
  \;\leq\;
  C_1\,\log\Bigl(\frac{C_2}{\delta}\Bigr),
\end{equation}
for some constants $C_1,C_2>0$ depending only on $A,B,v$, but not on
$Q$.

By uniform boundedness, $\|g\|_{L_1(Q)} \leq \|g\|_{L_2(Q)}$ for all
$g \in \mathcal{V} - \mathcal{V}$ and all $Q$. Hence
\begin{equation}\label{eq:VC-entropy-L1}
  N\bigl(\delta,\mathcal{V},L_1(Q)\bigr)
  \;\leq\;
  N\bigl(\delta,\mathcal{V},L_2(Q)\bigr),
\end{equation}
and \eqref{eq:VC-entropy-L2-simplified} yields a corresponding $L_1(Q)$
entropy bound. In particular, for each $\epsilon_\ell \in (0,1)$ and each $Q$,
I may choose a minimal $L_1(Q)$-covering of $\mathcal{V}$ at radius
$\epsilon_\ell/2$, that is, a finite set
$\{v_{1,\ell},\dots,v_{n_\ell,\ell}\}\subset\mathcal{V}$ such that
\[
  \mathcal{V} \subset \bigcup_{i=1}^{n_\ell}
  B_{L_1(Q)}\bigl(v_{i,\ell},\epsilon_\ell/2\bigr),
\]
and $n_\ell = N(\epsilon_\ell/2,\mathcal{V},L_1(Q))$. By
\eqref{eq:VC-entropy-L1} and \eqref{eq:VC-entropy-L2-simplified}, there exist
constants $C,p>0$ depending only on the VC characteristics of $\mathcal{V}$
such that
\[
  n_\ell \;\leq\; C\,\epsilon_\ell^{-p},
\]
which is \eqref{eq:V-nl-polynomial}.

By the covering property, for every $v \in \mathcal{V}$ there exists
$i \in \{1,\dots,n_\ell\}$ such that
$\|v - v_{i,\ell}\|_{L_1(Q)} \leq \epsilon_\ell/2$. In particular,
\[
  \sup_{v \in \mathcal{V}} \min_{1 \leq i \leq n_\ell}
  \|v - v_{i,\ell}\|_{L_1(Q)} \;\leq\; \epsilon_\ell/2,
\]
and hence \eqref{eq:V-epsilon-net} holds with $\epsilon_\ell$ and a
possibly enlarged constant $C$ (absorbing the factor $2$). Defining the sets
$E_{i,\ell}$ by nearest-center assignment,
\[
  E_{i,\ell}
  \;\coloneqq\;
  \Bigl\{ v\in\mathcal{V} :
    \|v - v_{i,\ell}\|_{L_1(Q)}
    = \min_{1 \leq j \leq n_\ell} \|v - v_{j,\ell}\|_{L_1(Q)}
  \Bigr\},
\]
with deterministic tie-breaking, yields a partition
$\{E_{i,\ell}\}_{i=1}^{n_\ell}$ of $\mathcal{V}$, and the covering property
implies
\[
  \sup_{v \in E_{i,\ell}} \|v - v_{i,\ell}\|_{L_1(Q)}
  \;\leq\; \epsilon_\ell,
  \qquad i=1,\dots,n_\ell,
\]
which is \eqref{eq:partition-approx-property}. The entropy bound
\eqref{eq:VC-entropy-L2} is uniform over all probability measures $Q$, so the
constants $C,p$ in \eqref{eq:V-nl-polynomial} may be chosen independent of
$Q$. In particular, the conclusions hold uniformly for $Q = R_\gamma$ with
$\gamma>0$. This proves the lemma.
\end{proof}

\section{An Implementation and Numerical Illustration}\label{Section - Implementation}
This section presents a geometric implementation of the sieve M-estimation procedure, and illustrates its performance numerically relative to the empirical Sinkhorn divergence. 
\subsection{Implementation}
\par The approximation scheme of~\cite{Tabri-MOOR-2025} requires a specification of $E_{i,\ell}$ for $i=1,\ldots, n_\ell$ satisfying the accuracy~(\ref{eq - Disc Accuracy}), which is a partition of the set of moment functions $\overline{\mathcal{V}}$. His scheme also allows the practitioner to select any set of $v_i\in E_{i,\ell}$ for $i=1,\ldots,n_\ell$, as the results of Proposition~\ref{Thm - Computation} are uniform in their choice. This section puts forward a specification of said partition and an approach to selecting the $v_i\in E_{i,N}$ based on the fact that $B:=\left\{v\in\mathcal{V}: \int_{\Omega}v\,dP_\gamma=0\right\}$ coincides with $\mathcal{V}$.

\par To elucidate, first note that
$\mathcal{V}=\bigcup_{i=1}^4\mathcal{V}_i$, where
\begin{align*}
\mathcal{V}_1 & =\left\{\left(F_X(x^\prime)-1\left[x\preceq x^\prime\right]\right)1\left[y\in\mathcal{Y}\right],x^\prime\in\mathcal{X}\right\},\quad\mathcal{V}_2=\left\{\left(1\left[x\preceq x^\prime\right]-F_X(x^\prime)\right)1\left[y\in\mathcal{Y}\right],x^\prime\in\mathcal{X}\right\}\\
\mathcal{V}_3 & =\left\{\left(F_Y(y^\prime)-1\left[y\preceq y^\prime\right]\right)1\left[x\in\mathcal{X}\right],y^\prime\in\mathcal{Y}\right\},\quad \mathcal{V}_4 =\left\{\left(1\left[y\preceq y^\prime\right]-F_Y(y^\prime)\right)1\left[x\in\mathcal{X}\right],y^\prime\in\mathcal{Y}\right\}.
\end{align*}
In this notation, observe that $\mathcal{V}_2=-\mathcal{V}_1$ and $\mathcal{V}_4=-\mathcal{V}_3$. Now, given $\epsilon_\ell$, I can always find finite partitions of $\mathcal{X}$ and $\mathcal{Y}$, given by $\{\mathcal{X}_j,j\leq n_x\}$ and $\{\mathcal{Y}_j,j\leq n_y\}$, respectively, such that
\begin{align}
\left|F_X(x)-F_X(x^\prime)\right| & \leq \epsilon_\ell/8\quad\text{and}\quad \left|R_{\gamma}\left(x,\right)-R_{\gamma,X}(x^\prime)\right|\leq \epsilon_\ell/8\quad\forall x,x^\prime\in \mathcal{X}_j\quad\text{and}\\
\left|F_Y(y)-F_Y(y^\prime)\right| & \leq \epsilon_\ell/8\quad\text{and}\quad \left|R_{\gamma,Y}(y)-R_{\gamma,Y}(y^\prime)\right|\leq \epsilon_\ell/8\quad\forall y,y^\prime\in \mathcal{Y}_i,
\end{align}
hold, for $j=1,\ldots,n_x$ and $i=1,\ldots,n_y$. The notation $R_{\gamma,X}(\cdot)$ and $R_{\gamma,Y}(\cdot)$ denote the distribution functions of the $X$-margin and $Y$-margin, respectively, of the joint distribution $R_\gamma$.

\par The partition has two parts, once for $\mathcal{V}_1$ and $\mathcal{V}_2$, and one for $\mathcal{V}_3$ and $\mathcal{V}_4$. Consider the following sets
$E_{X,1,j}=\left\{\left(F_X(x^\prime)-1\left[x\preceq x^\prime\right]\right)1\left[y\in\mathcal{Y}\right],x^\prime\in\mathcal{X}_j\right\}$, $E_{X,2,j}=\left\{\left(1\left[x\preceq x^\prime\right]-F_X(x^\prime)\right)1\left[y\in\mathcal{Y}\right],x^\prime\in\mathcal{X}_j\right\}$,
$E_{Y,1,i}=\left\{\left(F_Y(y^\prime)-1\left[y\preceq y^\prime\right]\right)1\left[x\in\mathcal{X}\right],y^\prime\in\mathcal{Y}_i\right\}$, $E_{Y,2,i} =\left\{\left(1\left[y\preceq y^\prime\right]-F_Y(y^\prime)\right)1\left[x\in\mathcal{X}\right],y^\prime\in\mathcal{Y}_i\right\}$, for $j=1,\ldots,n_x$ and $i=1,\ldots,n_y$. Then, I have partitions of the disjoint sets that comprise $\mathcal{V}$ and satisfy the approximation's construction: $\mathcal{V}_1=\cup_{j=1}^{n_x}E_{X,1,j}$ and $\mathcal{V}_2 =\cup_{j=1}^{n_x}E_{X,2,j}$, $\mathcal{V}_2 =\cup_{j=1}^{n_x}E_{X,2,j}$, $\mathcal{V}_3 = \cup_{i=1}^{n_y}E_{Y,1,i}$, and $\mathcal{V}_4 = \cup_{i=1}^{n_y}E_{Y,2,i}$. In consequence, $n_\ell=2n_x+2n_y$.

\par Now I shall make use of the fact that $B=\mathcal{V}$, holds, in selecting the functions $v_i$ to implement a moment \emph{equality} using two such inequalities. This can be done as follows: for each selected $v\in E_{X,1,j}$ and $v^\prime\in E_{Y,1,i}$, select $-v\in E_{X,2,j}$ and $-v^\prime\in E_{Y,2,i}$, and do this for $j=1,\ldots,n_x$ and $i=1,\ldots,n_y$. In consequence, I have reduced the number of choice variables in the approximating Fenchel dual programs by re-writing it to recognize moment equality constraints. In this case and using the result of Proposition~\ref{prop - alpha UB}, the problem~(\ref{eq - finite program 1}) becomes
\begin{align}
\inf & \left\{\int_{\Omega}G_\ell(\omega,\tau)\,dR_\gamma(\omega):\tau\in[-1,1]^{n_x+n_y},\text{and}\,\sum_{i=1}^{n_x+n_y}\tau_i\in[-1,1] \right\}\quad\text{where}\label{eq - finite program reduced 1}\\
 G_\ell(\omega,\tau)& =\exp\left\{\gamma\|c\|_\infty\sum_{i=1}^{n_x+n_y}\tau_i\,v_i(\omega)\right\},\label{eq - finite program reduced 2}
\end{align}
reducing the number of choice variables from $2n_x+2n_y$ to $n_x+n_y$ and constraining them to the box $[-1,1]^{n_x+n_y}$. With this formulation $n_\ell=n_x+n_y$, and function class $$\left\{G_\ell(\cdot,\tau),\tau\in[-1,1]^{n_\ell}\,\text{and}\,\sum_{i=1}^{n_\ell}\tau_i\in[-1,1]\right\}$$ is still VC and uniformly bounded by $e^{\gamma\kappa\|c\|_\infty}$ for each $\ell$, so that the results of Theorem~\ref{Thm Stochastic Opt Sieve} apply to the SAA version of~(\ref{eq - finite program reduced 1}).

\par In practice, we will consider large values of $\gamma$ (e.g., $\gamma\geq100$), which can create numerical challenges in solving the SAA version of~(\ref{eq - finite program reduced 1}), 
\begin{align}\label{eq - SAA Comp unStable}
\inf\left\{\frac{1}{N_\ell}\sum_{j=1}G_\ell(\omega_j,\tau):\tau\in [-1,1]^{n_\ell}\,\text{and}\,\sum_{i=1}^{n_\ell}\tau_i\in[-1,1]\right\},
\end{align}
for a random sample of size $N_\ell$ from $R_\gamma$. A practical approach to addressing this challenge is to rescale the objective function in~(\ref{eq - SAA Comp unStable}) by dividing it by $e^{\gamma\kappa\|c\|_\infty}$, and then to apply logarithms to quash the impact of large $\gamma$ in the objective function~(\ref{eq - finite program reduced 2}), yielding 
\begin{align}\label{eq - SAA Comp Stable}
\inf\left\{-\gamma\kappa\|c\|_\infty+\log\frac{1}{N_\ell}\sum_{j=1}G_\ell(\omega_j,\tau):\tau\in [-1,1]^{n_\ell}\,\text{and}\,\sum_{i=1}^{n_\ell}\tau_i\in[-1,1]\right\},
\end{align}
Because the logarithm is strictly positive monotonic transformation, the solutions of the problems~(\ref{eq - SAA Comp Stable}) and~(\ref{eq - SAA Comp unStable}) coincide. Furthermore, rescaling the objective function in this way implies $$\left\{e^{-\gamma\kappa\|c\|_\infty}\,G_\ell(\cdot,\tau),\tau\in[-1,1]^{n_\ell}\,\text{and}\,\sum_{i=1}^{n_\ell}\tau_i\in[-1,1]\right\}$$ is still VC and now uniformly bounded by 1 for each $\ell$. Hence, to calculate $\hat{\vartheta}_\ell$, I must first add $\gamma\kappa\|c\|_\infty$ to the optimal value~(\ref{eq - SAA Comp Stable}) and then apply the exponential function to it. Finally, with this setup, I can implement the limit~(\ref{eq - entropy bound}) as
\begin{align}\label{eq - imp entropy limit}
\lim_{\ell\rightarrow\infty}\frac{\log n_\ell}{N_\ell}=0.
\end{align}

\subsection{Numerical Illustration}
This section considers a toy example that is useful for establishing proof of concept of my approximation scheme and its implementation. I specify $X\sim U[0,1]$ and $Y~\sim U[0,2]$, and $c(x,y)=\frac{1}{2}(x-y)^2$. This specification satisfies the conditions of Assumption~\ref{Assump -Primities OTP}.

\par In this setting, the optimal transport value is the squared Wasserstein distance between $P_X$ and $P_Y$ (\citealp{villani2009optimal}), which is given by
\begin{align}\label{eq - OTP 2-Wasser Uniforms}
W^2_2(P_X,P_Y)=\inf_{P\in\Pi(P_X,P_Y)}\frac{1}{2}\int_{\mathcal{X}\times\mathcal{Y}}(x-y)^2\,dP.
\end{align}
The value and the solution of~(\ref{eq - OTP 2-Wasser Uniforms}) can be calculated in closed-form. In particular, $W^2_2(P_X,P_Y)=1/6\approx 0.166$ and the solution is the singular distribution supported on $\{(x,y): y=2x\}$ having CDF
\begin{align}
F_{X,Y}(x,y)
= \begin{cases}
0, & x,y < 0, \\[6pt]
\min\{x,\,y/2\}, & 0 \leq x\leq 1,\,0\leq y\leq2, \\[6pt]
1, & x>1,\,\,y>2.
\end{cases}
\end{align}

\par The EOT value with $\gamma=100$ is approximately 0.1846, and I have compared the finite-sample performance of the EOT value estimators based on the sieve M-estimator and empirical Sinkhorn divergence. I have studied their performance using 1000 Monte Carlo draws from $R_\gamma$ to estimate the absolute deviation of their means from the target, 0.1846. I have set the sample size for both estimators to be the optimal one for the sieve-based estimator, given by display~(\ref{eq - Optimal Sample size}) in Section~\ref{Section Discussion}, and I repeat it here for convenience: $N_\ell= 2\epsilon_\ell^{-2}\log n_\ell$. The sieve M-estimator is $\hat{W}^2_2(P_X,P_Y)=-\gamma^{-1}\log\hat{\vartheta}_{\ell}+\gamma^{-1}\log a_\gamma$, where I have calculated $\hat{\vartheta}_{\ell}$ via the optimization problem~(\ref{eq - SAA Comp Stable}). Table~\ref{Table - uniforms promitives and output} reports the primitives of the optimization problem~(\ref{eq - SAA Comp Stable}). Table~\ref{Table - uniforms promitives and output} reports the primitives of the optimization problem~(\ref{eq - SAA Comp Stable}). I have obtained the values 0.0254 and 0.0375 for the sieve M-esitmator and the empirical Sinkhorn divergence, respectively. While this numerical result indicates the sieve M-estimator is less biased than its empirical Sinkhorn counterpart, it is more informative to report the boxplots of their MC estimates. The boxplots approximate the sampling distributions of estimators, and they are a powerful visual tool for illustrating finite-sample properties of the estimators.

\begin{table}[pt]
\caption{Primitives}\label{Table - uniforms promitives and output}
\centering
\bigskip
\begin{tabular}{c*{5}{c}}
 $\gamma$ & $\kappa$  & $\epsilon_\ell$    & $n_\ell$        & $N_\ell$   \\
\hline 
100 & 1        & 0.1            &  160        & 1015           \smallskip\\

\end{tabular}
\end{table} 

\begin{figure}[pt]
\centering
\includegraphics[width=0.8\textwidth]{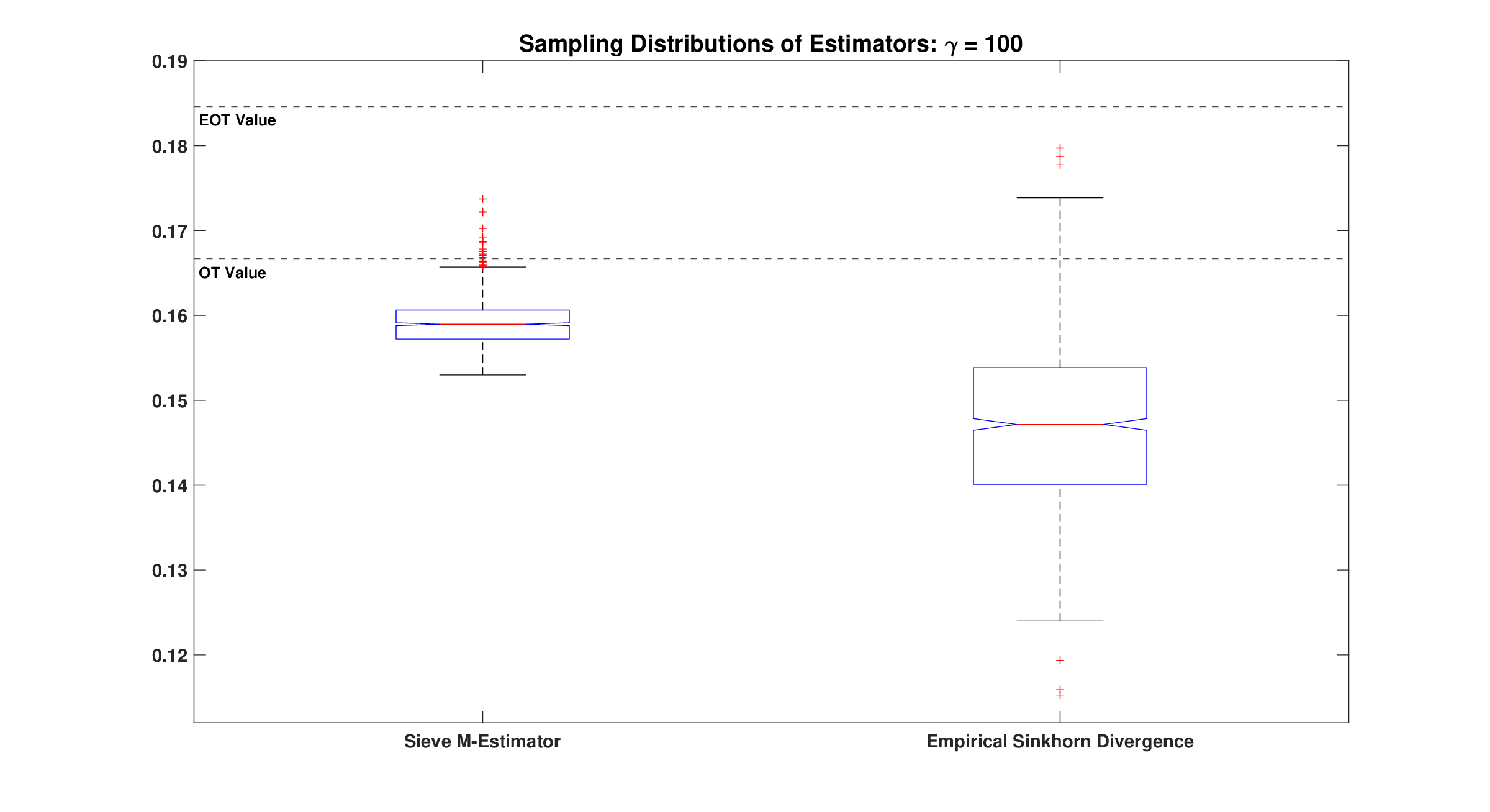}
\caption{Boxplots of simulated estimates}\label{Fig - 1}
\end{figure}

\par Figure~\ref{Fig - 1} reports the boxplots. Both estimators are centered below the OT benchmark, indicating a pronounced finite-sample downward bias in this high-regularization regime. This behavior contrasts with the population ordering, where the EOT value exceeds the unregularized OT cost, and highlights the impact of finite-sample effects and discretization error on the empirical objective. Relative to the empirical Sinkhorn divergence, the sieve M-estimator exhibits reduced variability and lower bias, suggesting improved stability at the theoretically prescribed sieve dimension and sample size.

\par The improved finite-sample performance of the sieve M-estimator relative to the empirical Sinkhorn estimator can be understood through the structure of their respective feasible sets. 
The population OT and EOT problems are characterized by exact marginal constraints, which may be equivalently expressed as an infinite collection of moment conditions.
The sieve M-estimator replaces this infinite system by a finite, accuracy-controlled set of moment constraints indexed by $\epsilon_\ell$, thereby constructing a deterministic approximation to the population constraint set.
As a result, even in finite samples, the sieve estimator optimizes over a feasible region that remains closely aligned with the population OT and EOT problems.

\par In contrast, the empirical Sinkhorn estimator enforces marginal constraints through empirical measures, leading to a feasible set that is random and subject to sampling fluctuations.
These fluctuations introduce additional variability and bias into the estimated transport cost, particularly in finite samples.
This distinction helps explain why, in the simulation results, the sieve M-estimator is both more concentrated and systematically closer to the population OT and EOT benchmark values than the empirical Sinkhorn estimator at the same sample size.

\end{document}